\documentclass[11pt]{amsart}
\usepackage[left=3.cm, right=3.cm, top=3.5cm, bottom=3.5cm]{geometry}
\usepackage{amssymb}
\usepackage{cite}
\newtheorem{thm}{Theorem}[section]
\newtheorem{cor}[thm]{Corollary}
\newtheorem{conj}[thm]{Conjecture}
\newtheorem{prop}[thm]{Proposition}
\newtheorem{lem}[thm]{Lemma}

\theoremstyle{definition}

\newtheorem{example}[thm]{Example}
\newtheorem{definition}[thm]{Definition}
\newtheorem{remark}[thm]{Remark}

\newcommand{\Diff}{\operatorname{Diff}}
\newcommand{\lspan}{\operatorname{span}}

\newcommand{\cU}{\mathcal{U}}

\newcommand{\cA}{\mathcal{A}}

\newcommand{\cL}{\mathcal{L}}
\newcommand{\cP}{\mathcal{P}}

\newcommand{\cQ}{\mathcal{Q}}

\newcommand{\Nset}{\mathbb{N}}
\newcommand{\Zset}{\mathbb{Z}}
\newcommand{\Rset}{\mathbb{R}}
\newcommand{\Cset}{\mathbb{C}}

\newcommand{\lp}{\left(}
\newcommand{\rp}{\right)}

\newcommand{\rL}{\mathrm{L}}
\newcommand{\Wr}{\operatorname{Wr}}

\newcommand{\ord}{\operatorname{ord}}
\newcommand{\Ann}{\operatorname{Ann}}

\newcommand{\hT}{\hat{T}}
\newcommand{\hL}{\hat{L}}
\newcommand{\hW}{\hat{W}}

\newcommand{\hb}{\hat{b}}
\newcommand{\hw}{\hat{w}}
\newcommand{\hp}{\hat{p}}
\newcommand{\hq}{\hat{q}}
\newcommand{\hy}{\hat{y}}
\newcommand{\hr}{\hat{r}}

\newcommand{\tH}{{\tilde{H}}}

\newcommand{\tM}{{\tilde{M}}}

\newcommand{\hmu}{\hat{\mu}}

\newcommand{\heta}{\hat{\eta}}

\newcommand{\WH}{W_{\text{H}}}
\newcommand{\WL}{W_{\text{L}}}
\newcommand{\WJ}{W_{\text{J}}}

\newcommand{\rms}{{\mathrm{s}}}

\newcommand{\talpha}{\tilde{\alpha}}

\newcommand{\tr}{\tilde{r}}

\newcommand{\ty}{\tilde{y}}
\newcommand{\tlambda}{\tilde{\lambda}}
\newcommand{\teta}{\tilde{\eta}}
\newcommand{\tmu}{\tilde{\mu}}
\newcommand{\tnu}{\tilde{\nu}}
\newcommand{\tzeta}{\tilde{\zeta}}
\newcommand{\tq}{\tilde{q}}

\newcommand{\tL}{\tilde{L}}
\newcommand{\tT}{\tilde{T}}
\newcommand{\tcU}{\tilde{\cU}}

\newcommand{\rmin}{\rho_{\min}}
\newcommand{\TB}{T_{\mathrm{B}}}
\newcommand{\WB}{W_{\mathrm{B}}}
\newcommand{\TS}{T_{\mathrm{s}}}

\begin{document}

\title[]{A Bochner type characterization theorem for exceptional orthogonal polynomials}

\author{M\textordfeminine  \'Angeles Garc\'ia-Ferrero}
\address{Instituto de Ciencias Matem\'aticas (CSIC-UAM-UC3M-UCM),  C/ Nicolas Cabrera 15, 28049 Madrid, Spain.}

\author{David G\'omez-Ullate}
\address{Departamento de F\'isica Te\'orica II, Universidad Complutense de
Madrid, 28040 Madrid, Spain.}
\address{Instituto de Ciencias Matem\'aticas (CSIC-UAM-UC3M-UCM),  C/ Nicolas Cabrera 15, 28049 Madrid, Spain.}
\author{Robert Milson}
\address{Department of Mathematics and Statistics, Dalhousie University,
Halifax, NS, B3H 3J5, Canada.}
\email{mag.ferrero@icmat.es, david.gomez-ullate@icmat.es,  rmilson@dal.ca}
\begin{abstract}
  It was recently conjectured that every system of exceptional
  orthogonal polynomials is related to classical orthogonal
  polynomials by a sequence of Darboux transformations.  In this paper
  we prove this conjecture, which paves the road to a complete
  classification of all exceptional orthogonal polynomials.  In some
  sense, this paper can be regarded as the extension of Bochner's
  result for classical orthogonal polynomials to the exceptional
  class.  As a supplementary result, we derive a canonical form for
  exceptional operators based on a bilinear formalism, and
  prove that every exceptional operator has trivial monodromy at all
  primary poles.
\end{abstract}

\maketitle
\tableofcontents
\section{Introduction}

Exceptional orthogonal polynomials are complete systems of orthogonal
polynomials that satisfy a Sturm-Liouville problem. They differ from
the classical families of Hermite, Laguerre and Jacobi in that there
are a finite number of exceptional degrees for which no polynomial
eigenfunction exists. The total number of gaps in the degree sequence
is the \textit{codimension} of the exceptional family.  As opposed to
their classical counterparts \cite{Szego1939, Ismail2005}, the
differential equation contains rational instead of polynomial
coefficients, yet the eigenvalue problem has an infinite number of
polynomial eigenfunctions that form the basis of a weighted Hilbert
space.  Because of the missing degrees, exceptional polynomials
circumvent the strong limitations of Bochner's classification theorem,
which characterizes classical Sturm-Liouville orthogonal polynomial
systems \cite{Bochner1929,Lesky1962}.

The recent development of exceptional polynomial systems has received
contributions both from the mathematics community working on
orthogonal polynomials and special functions, and from mathematical
physicists. Among the physical applications, exceptional polynomial
systems appear mostly as solutions to exactly solvable quantum
mechanical problems, describing both bound states
\cite{Gomez-Ullate2014a,Dutta2010,Grandati2011b,Grandati2011c,Lévai2010,Odake2009a,QUESNE2011,Quesne2009,Sesma2010}
and scattering amplitudes
\cite{Ho2014a,Yadav2015,Yadav2013,Yadav2013a}. But there are also
connections with super-integrability \cite{Post2012,Marquette2013a}
and higher order symmetry algebras
\cite{Marquette2014,Marquette2013,Marquette2013c}, diffusion equations
and random processes \cite{Ho2011b,Ho2014,CHOU2013}, quantum
information entropy \cite{Dutta2011}, exact solutions to Dirac
equation \cite{Schulze-Halberg2014} and finite-gap potentials
\cite{Hemery2010}.

Some examples of exceptional polynomials were investigated back in the
early 90s, \cite{Dubov1994} but their systematic study started a few
years ago, where a full classification was given for codimension one,
\cite{Gomez-Ullate2009a,Gomez-Ullate2010c}. Soon after that, Quesne
recognised the role of Darboux transformations in the construction
process and wrote the first codimension two
examples,\cite{Quesne2008}, and Odake \& Sasaki showed families for
arbitrary codimension, \cite{Odake2009a,Odake2010}. The role of
Darboux transformations was further clarified in a number of works,
\cite{Gomez-Ullate2010b,Sasaki2010,QUESNE2011}, and the next
conceptual step involved the generation of exceptional families by
multiple-step or higher order Darboux transformations, leading to
exceptional families labelled by
multi-indices,\cite{Gomez-Ullate2012,Odake2011,Grandati2012a}.  Other
equivalent approaches to build exceptional polynomial systems have
been developed in the physics literature, using the prepotential
approach \cite{Ho2011a} or the symmetry group preserving the form of
the Rayleigh-Schr\"odinger equation \cite{Grandati2012}, leading to
rational extensions of the well known solvable potentials.

In the mathematical literature, two main questions have centered the
research activity in relation to exceptional polynomial systems:
describing their mathematical properties and achieving a complete
classification.  Among the mathematical properties, the study of their
zeros deserve particular attention. Zeros of exceptional polynomials
are classified into two classes: \textit{regular zeros} which lie in
the interval of orthogonality and \textit{exceptional zeros}, which
lie outside this interval. Their interlacing, asymptotic behaviour,
monotonicity as a function of parameters and electrostatic
interpretation have been investigated in a number of
works,\cite{Dimitrov2014,Gomez-Ullate2013,Ho2012,Horvath2015,Kuijlaars2014},
but there are still open problems in this direction.  

A fundamental object in the theory of orthogonal polynomials is the
recurrence relation. Classical orthogonal polynomials have a three
term recurrence relation, but exceptional polynomial systems have
recurrence relations whose order is higher than three. There is a set
of recurrence relations of order $2N+3$ where $N$ is the number of
Darboux steps \cite{Odake2013a,Gomez-Ullate2014a} with coefficients
that are functions of $x$ and $n$, and another set of recurrence
relations whose coefficients are just functions of $n$ (as in the
classical case) and whose order is $2m+3$ where $m$ is the
codimension, \cite{Miki2015,Duran2015a,Odake2015}. While the former
relations are generally of lower order and thus more convenient for an
efficient computation, the latter are more amenable to a theoretical
interpretation in terms of the usual theory of Jacobi matrices and
bispectrality. The spectral theoretic aspects of exceptional
differential operators were first addressed in
\cite{Everitt2008,Everitt2008a} and developed more recently in a
series of papers \cite{Liaw2015,Liaw2014,Liaw2015b}.

The quest for a complete classification of exceptional polynomials has
been fundamental problem that is now close to being solved, and the
results in the present paper are a key step towards this goal. The
first attempts to classify exceptional polynomial systems proceeded by
increasing codimension. Codimension one systems were classified in
\cite{Gomez-Ullate2009a} and they included just one $X_1$-Laguerre and
one $X_1$-Jacobi family. The classification for codimension two was
performed in \cite{Gomez-Ullate2012a}, based on an exhaustive
case-by-case enumeration of invariant flags under a given symmetry
group.  Due to the combinatorial growth of complexity with increasing
codimension, this original approach proved to be unfeasible for the
purpose of achieving a complete classification. However, a fundamental
idea towards the full classification was also launched in
\cite{Gomez-Ullate2012a}, namely that every exceptional polynomial
system can be obtained from a classical system by applying a finite
number of Darboux transformations. More precisely, the following
conjecture was formulated:
\begin{conj}\label{conj}[G\'omez-Ullate, Kamran, Milson 2012]
Every exceptional orthogonal polynomial system of codimension $m$ can be
obtained by applying a sequence of at most $m$ Darboux transformations to a classical orthogonal polynomial system. 
\end{conj}
If the conjecture holds, then the program to classify exceptional
polynomial systems becomes constructive: start from the three
classical systems of Hermite, Laguerre and Jacobi and apply all
possible Darboux transformations to describe the entire exceptional
class. It should be stressed that only rational Darboux
transformations need to be considered, i.e. those that map polynomial
eigenfunctions into polynomial eigenfunctions, and this type of
transformations are well understood and catalogued, and they are
indexed by sequences of integers. This constructive approach has
already been used to generate large classes of exceptional polynomial
systems. The most general class obtained in this way can be labeled
by two sets of indices or partitions (for the Laguerre and Jacobi
classes) \cite{Duran2014a} or just one set (for the Hermite class)
\cite{Gomez-Ullate2014a,Felder2012} which can be conveniently represented in a
Maya diagram \cite{Takemura2014}, a representation that takes
naturally into account a number of equivalent sets of indices that
lead to the same exceptional system, \cite{Odake2014,Gomez-Ullate2016b}. However, the
question of whether this list contains all exceptional polynomials
remained open.

In all examples known so far, the weight for the exceptional system
$W(z)$ is a rational modification of a classical weight $W_0(z)$
having the following form:
\begin{equation}\label{eq:XW}
  W(z)=\frac{W_0(z)}{\eta(z)^2},
\end{equation}
where $\eta(z)$ is a polynomial in $z$ (a Wronskian-like determinant)
whose degree coincides with the codimension of the system. In this
paper we prove that this is indeed the case for any possible
exceptional polynomial system.

One important point remains, namely that of ensuring that the
transformed weight gives rise to a well defined spectral problem,
which we shall refer to as the \textit{weight regularity}
problem. This means studying the sequence of Darboux transformations
and the range of parameters for which:
\begin{enumerate}
\item[i)] the weight has the right asymptotic behaviour at the endpoints 
\item[ii)]  $\eta(z)$  has no zeros inside the interval of orthogonality.
\end{enumerate}
The regularity problem has been solved for the exceptional Hermite
class \cite{Gomez-Ullate2014a,Duran2014} based on results by Krein
\cite{Krein1957}, and Adler \cite{Adler1994}, and also for the
Laguerre class, \cite{Duran2015b}, using a remarkable correspondence
between exceptional polynomials and discrete Krall type polynomials,
\cite{Duran2014a}.

The main result of this paper is the following theorem, which is
essentially a proof of Conjecture \ref{conj}, albeit without a bound
on the number of Darboux steps.

\begin{thm}\label{thm:XOPS} Every exceptional orthogonal polynomial
  system can be obtained by applying a finite sequence of Darboux
  transformations to a classical orthogonal polynomial system.
\end{thm} 

The essential consequence of this result is that it places on safe
ground the constructive approach to the full classification described
above. The strategy of the proof involves several steps.

First, we establish a number of factorization results for second-order
order differential operators with rational coefficients.  In
particular, in Section \ref{sec:Darbtrans} we show that every
higher-order intertwiner can be factorized into a composition of first
order operators, each of them corresponding to a one-step, rational
Darboux transformation.

We introduce exceptional operators in Section \ref{sec:exop}, and
prove a fundamental theorem that relates the codimension to the sum of
certain integer indices at the poles of the operator.  Next, in
section \ref{sec:struct} we prove that every exceptional operator
admits a canonical formulation as a bilinear relation between two
polynomials.  The key technical tools are some results on the local
behaviour of solutions around the singular points of the differential
equations corresponding to exceptional operators.  A further key step
is the demonstration that an exceptional operator has trivial
monodromy at almost every point $\zeta\in\mathbb C$. This result was
already known for the exceptional Hermite class
\cite{Gomez-Ullate2014a,Oblomkov1999}, and we show that it can be
extended to a general exceptional operator. The connection between
trivial monodromy, bispectrality, Darboux transformations and the
solvable character of Schr\"odinger operators has been discussed in a
number of papers (see for instance
\cite{Duistermaat1986,Oblomkov1999,Veselov1993,Gibbons2009,Grunbaum1997,Veselov2001}
and the references therein), and the results in this paper are one
further piece of evidence of the close relationship among these
concepts.

In Section \ref{sec:L} we build on the structural properties of
exceptional operators to prove the existence of a higher order
intertwiner between any exceptional operator and a classical operator,
extending the proof given by Oblomkov \cite{Oblomkov1999} for the
rational extensions of the harmonic oscillator.

Finally the
proof of Theorem \ref{thm:XOPS} is given in Section \ref{sec:OPS}
making use of all the previous results. This section also contains
Theorem \ref{prop:Xweight} which states that the orthogonality weight
for any exceptional polynomial system has the form \eqref{eq:XW}.

\section{Preliminaries} \label{sec:prelim}

In this preliminary section we introduce some key definitions and
notation, and prove some essential results about second-order
differential operators with rational coefficients.  Let $\cQ=\Cset(z)$
denote the differential ring of univariate, complex-valued rational
functions and $\cP=\Cset[z]$ the subring of polynomials.  Let
$\cP_n\subset \cP,\; n\in \Nset$ denote the vector space of
polynomials of degree $\leq n$, and $\cP_n^*\subset \cP_n$ the subset
of polynomials whose degree is exactly equal to $n$.  Similarly, let
$\cQ_n$ denote the vector space of rational functions having degree
$\leq n$, where the degree of a rational function is defined to be the
difference of the degrees of the numerator and denominator.  

Let $\Diff(\cQ)=\Cset(z)[D_z]$ denote the ring of linear differential
operators with rational coefficients and $\Diff(\cP)=\Cset[z,D_z]$ the
subring of operators with polynomial coefficients.  Alternatively,
$\Diff(\cP)$ may be characterized as the subring of $\Diff(\cQ)$ that
preserves $\cP$.  When needed, will use $\Rset\cQ, \Rset\cP,
\Rset\cP_n$ to denote the corresponding real-valued subrings and
subspaces, and $\Diff(\Rset\cQ), \Diff(\Rset\cP)$ the corresponding
rings of real-valued differential operators.

For a sufficiently differentiable function
$y$, we let $D_z^j y=y^{(j)}(z)$ denote the $j^{\rm th}$ derivative of
$y(z)$ with respect to $z$. The notation $D_{zz}=D_z^2$ will also be
employed.  Let $\Diff_\rho(\cQ)$ denote the set of $\rho^{\text{th}}$
order differential operators; that is, operators of the form
\begin{equation}
  \label{eq:Lcoeffs}
  L = \sum_{j=0}^\rho a_j(z) D^j_z ,\quad a_j\in
  \cQ,\; a_\rho \neq 0,
\end{equation}
with action
\begin{equation}
  \label{eq:Lyaction}
  y\mapsto L[y] = \sum_{j=0}^\rho a_j(z) y^{(j)}(z) ,\quad y\in \cQ.  
\end{equation}

\begin{definition}
  We say that a function $\phi(z)$ is quasi-rational if its
  log-derivative 
  \[ D_z\Big[ \log \phi(z)\Big] = \frac{\phi'(z)}{\phi(z)}\] is a
  rational function of $z$.  
\end{definition}

For $T\in \Diff_2(\cQ)$, write
\begin{equation}\label{eq:Tgeneral} 
  T=p(z)D_{zz} + q(z) D_z + r(z),\quad p,q,r\in \cQ
\end{equation}
and define the quasi-rational functions
\begin{subequations}
  \label{eq:PWRdef}
\begin{align}
  \label{eq:Pdef}
  P(z) &= \exp\left(\int^z \frac{q(x)}{p(x)}dx\right),\\
  \label{eq:Wdef}
  W(z) &=\frac{P(z)}{p(z)},\\
  R(z) &= r(z) W(z).
\end{align}
\end{subequations}
Multiplying the eigenvalue relation $T[y]=\lambda y$ by $W(z)$ gives
an equivalent form as Sturm-Liouville type equation
\begin{equation}
  \label{eq:SLform}
  (P y')'+ R y = \lambda W y.
\end{equation}
\begin{prop}
  \label{prop:Tsym}
  The operator $T$ is formally symmetric with respect to $W$ in the
  sense that
  \begin{equation}
    \label{eq:Tsym}
    \int^z T[f](x) g(x) W(x) dx - \int^z T[g](x) f(x) W(x) dx = P(z)
    (f'(z) g(z)- f(z) g'(z)),
  \end{equation}
  where $f,g$ are sufficiently differentiable functions.
\end{prop}
\begin{proof}
  This follows by \eqref{eq:SLform} and integration by parts.
\end{proof}

\begin{definition}
  We say that two rational operators $T,\hT\in \Diff_2(\cQ)$ are
  gauge-equivalent if there exists a $\sigma \in \cQ$ such that
  \begin{equation}
    \label{eq:gauge-intertwine}
    \hT = \sigma T\sigma^{-1}.
  \end{equation}
  We will refer to $\sigma$ as the \emph{gauge-factor}.
\end{definition}
\begin{remark}
  Above we are using $\sigma$ to denote both a rational function, and the
  multiplication operator $y\mapsto \sigma y$.  The reason for the
  gauge-factor terminology is that the eigenvalue relation
  $T[y]=\lambda y$ is equivalent to the eigenvalue relation $\hT[\hy]
  = \lambda \hy$, with $\hy = \sigma y$.
\end{remark}
\begin{prop}
  \label{prop:gauge-equiv}
  Suppose that $T,\hT\in\Diff_2(\cQ)$ satisfy
  \eqref{eq:gauge-intertwine}.  Letting $p,q,r,\hp,\hq,\hr$ be the
  coefficients of $T$ and $\hT$ as per \eqref{eq:Tgeneral}, and
  $W, \hW$ the corresponding weights \eqref{eq:Wdef} we have the
  following transformation laws
  \begin{subequations}
    \label{eq:pqrWgaugelaw}
    \begin{align}
      p &= \hp \\      \label{eq:qgaugelaw}
      q &= \hq + \frac{ 2\sigma'}{\sigma}\,\hp\\ \label{eq:rgaugelaw}
      r &= \hr +  \frac{\sigma'}{\sigma}\, \hq+  \frac{\sigma''}{\sigma}\, \hp\\
      W &= \sigma^2 \hW.
    \end{align}
  \end{subequations}
\end{prop}

\section{Rational Darboux transformations}
\label{sec:Darbtrans}
The gauge-equivalence relation \eqref{eq:gauge-intertwine} is an
intertwining relation of second-order operators by a zero-order multiplication
operator.  Consideration of higher-order intertwining relations leads
naturally to the notion of a Darboux transformation.

\begin{definition}
  For $T\in \Diff_2(\cQ)$ a \textit{rational factorization} is a relation
  of the form
  \begin{equation}
    \label{eq:TBA}
    T=BA+\lambda_0,
  \end{equation}
  where $A,B\in \Diff_1(\cQ)$ and $\lambda_0\in \Cset$ is a constant.
  Given a rational factorization, we call the operator $\hT \in
  \Diff_2(\cQ)$ defined by
  \begin{equation}
    \label{eq:hTAB}
    \hT := AB+\lambda_0
  \end{equation}
  the \textit{partner operator} and say that $T\mapsto \hT$ is a
  \textit{rational Darboux transformation}.
\end{definition}

\begin{prop}
  \label{prop:intertwiner}
  Suppose that $T,\hT\in \Diff_2(\cQ)$ are related by a rational
  Darboux transformation.  Then, the following intertwining relations
  hold
  \begin{equation}
    \label{eq:TABintertwine}
     AT = \hT A,\qquad TB = B \hT.
  \end{equation}
\end{prop}
\begin{proof}
  This is a direct consequence of \eqref{eq:TBA} and \eqref{eq:hTAB}.
\end{proof}
\begin{remark}
  The intertwining relation \eqref{eq:TABintertwine} implies that
  the eigenvalue relation $T[y]=\lambda y$ is formally equivalent to
  the eigenvalue relation $\hT[\hy] = \lambda \hy$ where $\hy = A[y]$.
\end{remark}

\begin{definition}
  For $T\in \Diff_2(\cQ)$ and $\phi(z)$ quasi-rational, we will say that
   $\phi$  is a quasi-rational eigenfunction  of $T$ if
  \begin{equation}
    \label{eq:Tphi}
     T[\phi] = \lambda_0 \phi,\quad \lambda_0\in \Cset.
  \end{equation}
\end{definition}

We observe that to every quasi-rational eigenfunction $\phi$ of $T$ there corresponds a rational factorization, as shown by the following proposition.

\begin{prop}
  \label{prop:ratdarboux}
  For $T\in \Diff_2(\cQ)$, let $\phi(z)$ be a quasi-rational
  eigenfunction of $T$ with eigenvalue $\lambda_0$, and let $b(z)$ be an
  arbitrary, non-zero rational function.  Define rational functions
  \begin{equation}
    \label{eq:wdef}
    w= \frac{\phi'}{\phi},\qquad
    \hb= \frac{p}{b},\qquad
    \hw= -w-\frac{q}{p} + \frac{b'}{b},
  \end{equation}
  and  first order operators
  $A,B\in \Diff_1(\cQ)$ by
  \begin{equation}\label{eq:ABdef}
    A=b(z) (D_z - w(z)),\qquad  B=\hat b(z) (D_z-\hat w(z) ).
  \end{equation}
  With $A,B$ as above, the rational factorization relation
  \eqref{eq:TBA} holds.  Moreover, $w$ is a solution of the Ricatti
  equation 
  \begin{equation}
    \label{eq:pqrw}
    p(w '+ w^2)  +qw+ r=\lambda_0.
  \end{equation}
  Conversely, given a rational factorization \eqref{eq:TBA}, there
  exists a quasi-rational eigenfunction $\phi(z)$ with eigenvalue
  $\lambda_0$ and a rational $b(z)$ such that \eqref{eq:wdef},
  \eqref{eq:ABdef}, and \eqref{eq:pqrw} hold.
\end{prop}
\begin{proof}
  By \eqref{eq:Tphi} we have
  \[ \frac{p\phi''}{\phi} + \frac{q\phi'}{\phi} +r=\lambda_0.\] The
  Ricatti relation \eqref{eq:pqrw} follows immediately.  Applying
  \eqref{eq:wdef}, \eqref{eq:ABdef}, and \eqref{eq:pqrw} we have
  \begin{align*}
    (BA)[y] &= B[by'-bw y]\\
    &= \hb b y'' +( \hb b' -\hb bw- b \hb \hw) y'+ (w\hw b\hb
    - \hb (bw)' ) y\\
    &= p y'' + \left( \frac{pb'}{b} + p\left(\frac{q}{p} -
        \frac{b'}{b}\right)\right) y' + \left(p w \left(-w
        -\frac{q}{p} + \frac{b'}{b}\right)- p w \,\frac{b'}{b} - p
      w'\right) y\\
    &= py'' + qy' + (r-\lambda_0) y.
  \end{align*}

  We now prove the converse.  Suppose that \eqref{eq:TBA} holds.
  Let $b(z),w(z),\hb(z),\hw(z)$ be rational functions dictated by  the
  form \eqref{eq:ABdef}.
  Define the quasi-rational function
  \[ \phi(z) = \exp\left( \int^z\!\!  w(x)dx\right)\] so that
  $w=\phi'/\phi$. Then, \eqref{eq:Tphi} follows from \eqref{eq:TBA}.
  Expanding $(BA)[y]$, as above shows that
  \begin{align*}
    p = \hb b,\quad q = \hb b' - \hb b(w+\hw),\quad  r-\lambda_0 = w\hw
    b \hb - \hb (bw)'.
  \end{align*}
  From this \eqref{eq:wdef} and \eqref{eq:pqrw} follow immediately.
\end{proof}

The next proposition expresses the transformation law for the coefficients of a differential operator $T\in\Diff_2(\cQ)$ under a rational Darboux transformation specified by the rational functions $\phi$ and $b$.

\begin{prop}
  \label{prop:hWW}
  Suppose that $T,\hT\in \Diff_2(\cQ)$ are related by a rational
  Darboux transformation.  Then, the coefficients of $T$ and $\hT$ and the quasi-rational weights $W(z), \hW(z)$, as
  defined by \eqref{eq:Wdef}, are related by
  \begin{subequations}
  \begin{align}
   \label{eq:hp}
    \hat p &= p\\
    \label{eq:hq}
    \hq &= q+p'-  \frac{2b'}{b}\, p,\\
    \label{eq:hr}
    \hr &= 
     r+q'+wp'-\frac{b'}{b} \left(q + p'\right)+
     \left(2\,\left(\frac{ b'}{b}\right)^2 -\frac{b''}{b}+ 2
     w'\right) p\\
    \label{eq:hWW} \hW &= \frac{p}{b^2} W,
  \end{align}
  \end{subequations}
  where $b$ and $ w= (\log \phi)'$ are the rational functions defined in
  Proposition \ref{prop:ratdarboux}.  
\end{prop}
\begin{proof}
  By \eqref{eq:TBA}- \eqref{eq:ABdef}, $p(z)$ is the second-order
  coefficient of both $T, \hT$.  Let $\hq(z)\in \cQ$ be the
  first-order coefficient of $\hT$.  Relation \eqref{eq:hq} follows by
  \eqref{eq:wdef} \eqref{eq:ABdef} and \eqref{eq:hTAB}.  Applying
  \eqref{eq:PWRdef} and using \eqref{eq:hq} gives \eqref{eq:hWW}.
  Considering the hatted dual of \eqref{eq:pqrw} and applying
  \eqref{eq:wdef}  and \eqref{eq:hq} gives
  \begin{align*}
    \hr &= \lambda_0 - p (\hw' + \hw^2)- \hq \hw\\
    &= \lambda_0 - p\left(\left(-w-\frac{q}{p} + \frac{b'}{b}\right)' +
      \left(-w-\frac{q}{p} + \frac{b'}{b}\right)^2 \right) -
    \left(q+p'-
      \frac{2pb'}{b}\right)\left(-w-\frac{q}{p} + \frac{b'}	{b}\right),
  \end{align*}
  which simplifies to the expression shown in \eqref{eq:hr}.
\end{proof}

Next, we consider iterated rational Darboux transformations. In the
context of Schr\"odinger operators, these are known as higher-order
Darboux or Darboux-Crum transformations.  \cite{Crum1955}.
\begin{definition}
  \label{def:TDarbtrans}
  Let $\hT,T\in \Diff_2(\cQ)$ be second-order operators with rational
  coefficients.  We will say that $\hT$ is \textit{Darboux connected}
  to $T$ if there exists an operator $L\in \Diff(\cQ)$ 
  such that
  \begin{equation}
    \label{eq:ThTL}
    \hT L = LT.
  \end{equation}
\end{definition}
\begin{remark}
  Note that in the above definition the operator $L$ could have any order, and that gauge-equivalent operators \eqref{eq:gauge-intertwine}
  are Darboux connected by definition,  because they are related by a
  zero-th order intertwining relation.
\end{remark}

Rational Darboux transformations can also be iterated, a concept that leads
to the following definition.

\begin{definition}
  We will say that $\hT,T\in \Diff_2(\cQ)$ are connected by a
  factorization chain if there exist
  second-order operators $T_i\in \Diff_2(\cQ),\; i=0,1,\ldots, n$ with
  $T_0= T$ and $T_n=\hT$; first-order operators $A_i,B_i\in
  \Diff_1(\cQ),\; i=0,1,\ldots, {n-1}$, and constants $\lambda_i$ such
  that
  \begin{align}
    \label{eq:factorchain}
    T_i &= B_i A_i + \lambda_i,\qquad i=0,1,\ldots,{n-1}\\
    T_{i+1} &= A_i B_i + \lambda_i.
  \end{align}
\end{definition}


It is trivial to show that two operators connected by a factorization
chain are also Darboux connected.  The converse is also true
\cite{Oblomkov1999}[Theorem 1].  The just cited paper limits itself to
the case of operators in Schr\"odinger form, but we state and prove the
generalization for second-order operators with rational coefficients
using essentially the same argument.

\begin{thm}\label{thm:Darbconnected}
  Two rational operators $T,\hT\in \Diff_2(\cQ)$ are Darboux connected
  if and only if they are either gauge-equivalent, or they are
  connected by a factorization chain.
\end{thm}
\begin{proof}
  Suppose that $T$ and $\hT$ are connected by a factorization chain.
  By assumption,
  \[ T_{i+1} A_i  = A_i B_i A_i + \lambda_i A_i = A_i T_i,\qquad
  i=0,1,\ldots, n-1.\]
  It follows by induction that
  \[ T_{i+1} A_{i} \cdots A_0 = A_i \cdots A_0 T_0.\] 
  Therefore, \eqref{eq:ThTL} is satisfied with
  \[ L = A_{n-1} \cdots A_1 \cdot A_0. \]

  The proof of the converse is a modification of an argument given in
  \cite{Oblomkov1999}.  If $\ord L = 0$, then $T$ and $\hT$ are
  gauge-equivalent.  Thus,  suppose that \eqref{eq:ThTL} holds and
  that $\ord L \geq 1$. 

  Claim 1: no generality is lost if we assume that $L$
  does not have a right factor of the form $T-\lambda$.  Indeed,
  suppose that
  \[ L = \tL (T-\lambda),\quad \lambda \in \Cset.\] Since $T$ commutes
  with $T-\lambda$, it follows that
  \[ \hT \tL = \tL T \] is a lower order intertwining relation between
  $\hT$ and $T$.  Repeating this argument a finite number of times
  yields an intertwiner $L$ with the desired property.

  Claim 2: $T$ leaves $\ker L$ invariant. By relation \eqref{eq:ThTL},
  if $y\in\ker L$, then
  \[ L\big[T[y]\big] = \hT\big[L[y]\big] = 0,\]
  so $T[y]\in\ker L$ also.  

  Claim 3: if $T[y] = \lambda y$, then
  \begin{equation}
    \label{eq:LFG}
    L[y] = F(z,\lambda) y + G(z,\lambda) y',
  \end{equation}
  where $F,G$ are polynomial in $\lambda$ and rational in $z$.  By assumption,
  \[ T[y] = p(z) y''+ q(z) y' + r(z) y =\lambda y\] where $p(z),q(z),r(z)$ are
  rational in $z$. We have thus that
  \[ y'' = -\frac{q(z)}{p(z)} y' + \frac{\lambda -r(z)}{p(z)} y ,\]
  and hence a higher order derivative $y^{(k)}, \; k\geq 2$ can always
  be written as a linear combination of $y$ and $ y'$ with
  coefficients that are polynomial in $\lambda$ and rational in
  $z$.

  Since $\ker L$ is finite-dimensional and invariant with respect to
  $T$, let us choose an eigenvector $\phi\in \ker L$  of $T$ with eigenvalue
  $\lambda_0$.  It follows that
  \[ F(z,\lambda_0) \phi + G(z,\lambda_0)\phi' = 0,\] with $F,G$ defined above.

  Claim 4: $G(z,\lambda_0)$ is not identically zero.  If $\ord L=1$
  the claim is trivial.  For $\ord L\geq 2$ we argue by contradiction
  and suppose that $G(z,\lambda_0)\equiv 0$. Then,
  $F(z,\lambda_0)\equiv 0 $ also, which implies that
  \[ \ker (T-\lambda_0) \subset \ker L.\]  
  It can then easily be shown (see Theorem 1 in \cite{Oblomkov1999}
  and Section 5.4 of \cite{Ince2003}) that
  \[ L=\tL(T-\lambda_0),\]
  which violates the reducibility assumption established by Claim
  1. Claim 4 is proved.

  Thus, $G(z,\lambda_0)$ is \emph{not} identically zero, and
  therefore,
  \[ \frac{\phi'(z)}{\phi(z)} =
  -\frac{F(z,\lambda_0)}{G(z,\lambda_0)} \] is a rational
  function. Set $T_0=T$ and $L_0=L$. By Proposition
  \ref{prop:ratdarboux}, there exists a rational factorization
  \[ T= B_0A_0+\lambda_0 \] with $A_0[\phi] = 0$.  Since $\phi\in \ker
  L$ we also have a rational factorization
  \[ L = L_1 A_0 ,\quad L_1\in \Diff(\cQ).\] Setting
  \[ T_1 = A_0 B_0+\lambda_0\]
  we have
  \[ (\hT L_1 - L_1 T_1) A_0 = 0 \]
  which implies that 
  \[ \hT L_1 = L_1 T_1.\]
  
  Claim 5: $L_1$ has no right factors of the form $T_1-\lambda$.
  Suppose otherwise, so that
  \[ L_1 = \tL(T_1-\lambda).\] Then, setting $\tlambda = \lambda -
  \lambda_0$ we have
  \[ L= L_1 A_0 = \tL (A_0B_0-\tlambda) A_0 = \tL A_0(B_0A_0-\tlambda)
  = \tL A_0(T-\lambda),\] which again violates the irreducibility
  assumption of Claim 1.

  Continuing by induction, we have
  \[ \hT L_i = L_i T_i, \quad i=0,1,\ldots\] with $L_i$ reduced.
  Repeating the above argument, we construct rational factorizations
  \[ T_i = B_i A_i + \lambda_i,\quad T_{i+1} = A_i B_i +\lambda_i, \]
  so that
  \[ L = L_{i+1} A_i \cdots A_0 \]
  and
  \[ \hT L_{i+1} = L_{i+1} T_{i+1},\] and so that $L_{i+1}$ is reduced
  as per Claim 1.  This process terminates when $L_{i}$ is a
  first-order operator, because then we can take $L_i = A_i$, which
  gives $\hT= T_{i+1}$, and completes the factorization chain that
  connects $\hT$ and $T$.
\end{proof}
\begin{cor}
  The property of being Darboux connected is an equivalence relation
  on $\Diff_2(\cQ)$.
\end{cor}
\begin{proof}
  Reflexivity of the relation is self-evident.
  We need to prove that the Darboux connected relation possesses both
  symmetry and transitivity.  Suppose \eqref{eq:ThTL} holds.  If
  $L=\mu$ is zero-order then,
  \[ T \mu^{-1} = \mu^{-1} \hT, \] so that $T$ is Darboux connected to
  $\hT$.  If $\ord L\geq 1$ then $\hT$ and $T$ are related by a
  factorization chain.  By inspection of the definition, the property
  of being connected by a factorization chain is symmetric; one simply
  switches the $A_i$ and the $B_i$ and reverses the order of the
  factorization chain.

  Next suppose that
  \[ T_1 L_1 = L_1 T_2,\quad T_2 L_2 = L_2 T_3,\]
  where $T_1,T_2, T_3\in \Diff_2(\cQ)$ and $L_1, L_2\in \Diff(\cQ)$.
  Then, by associativity of operator composition,
  \[ T_1 L_1 L_2 = L_1 T_2 L_2 = L_1 L_2 T_3,\]
  so that $T_1$ is Darboux connected to $T_3$.
\end{proof}

\section{Exceptional operators and invariant polynomial subspaces}
\label{sec:exop}

\begin{definition}\label{def:exT}
  We will say that a second-order operator $T\in \Diff_2(\cQ)$ is
  \emph{exceptional} if $T$ has a polynomial eigenfunction for all but
  finitely many degrees.  The precise condition is that there exists a
  finite set of natural numbers $\{ k_1,\ldots, k_m \}\subset \Nset$
  such that for all $k\notin \{ k_1,\ldots, k_m \},$ there exists a
  $y_k\in \cP^*_k$ and a $\lambda_k\in \Cset$ such that
  \[ T[y_k]=\lambda_k y_k, \quad k\in \mathbb N-\{k_1,\dots,k_m\} \] 
  and such that no such polynomial exists
  if $k\in \{k_1,\ldots, k_m\}$.   We will refer to $k_1,\ldots, k_m$
  as the \textit{exceptional  degrees}.
\end{definition}

\begin{remark}
  Note that in the above definition of an exceptional differential
  operator, $m$ could be zero, i.e.  exceptional operators include
  classical operators as a special case. In the recent literature on
  this subject, the adjective \textit{exceptional} is usually reserved
  for the case $m>0$ to differentiate them from the classical
  ones. However, for the purpose of this paper it is convenient to
  handle the general class.  Thus, in order not to introduce further
  notation, we will stick to the term \textit{exceptional} in this
  wider context, hoping that no confusion will arise.
\end{remark}

As the following Proposition shows, no generality is lost by assuming
that an exceptional operator has rational coefficients.   Indeed, the
existence of just 3 linearly independent  polynomial eigenfunctions is
enough to conclude that  a second-order differential expression has
rational coefficients.
\begin{prop}
  \label{prop:Tcomponents}
  Let $T$ be a second-order differential operator as per
  \eqref{eq:Tgeneral} that maps three polynomials into polynomials.
Then the coefficients of $T$ are rational functions.
\end{prop}  
\begin{proof}
 The three conditions read
  \[ g_k = T[f_k] = p(z) f_k'' + q(z) f_k' + r(z)  f_k,\quad k=1,2,3.\]
  where $g_k$ and $f_k$ are polynomials. The coefficients $p,q,r$ are the unique solutions of the following linear equation:
  \begin{equation*}
    \begin{pmatrix}
      g_1 \\ g_2 \\ g_3
    \end{pmatrix}=
    \begin{pmatrix}
       f_1'' & f_1' & f_1\\
      f_2'' & f_2' & f_2\\
      f_3'' & f_3' & f_3
    \end{pmatrix}
    \begin{pmatrix}
      p\\q\\r
    \end{pmatrix}\,.
  \end{equation*}
\end{proof}

\label{sect:U}
\begin{definition}\label{def:U}
  For an exceptional operator $T\in \Diff_2(\cQ)$,  let
  $\cU\subset\cP$ denote the maximal invariant polynomial subspace, and
  $\nu$  the codimension of $\cU$ in $\cP$.
\end{definition}

Note that the polynomial subspace $\cU$ includes the span of all
polynomial eigenfunctions of $T$, but sometimes it can be larger (see
Remark \ref{rem:ss} and Example \ref{ex:1}).  It can also be
characterized in the following manner:

\begin{prop}
  \label{prop:Udef}
  An equivalent characterization of $\cU$ is
  \begin{equation}
    \label{eq:Udef}
    \cU=\{ y\in \cP \colon T^j[y]\in \cP \text{ for all } j\in \Nset \}.
  \end{equation}
\end{prop}
\begin{proof}
  Let $\cU'$ denote the subspace defined by the right side of
  \eqref{eq:Udef}.  For all $y\in \cU'$ we have $T[y]\in \cU'$ by
  definition. Hence, $\cU'$ is $T$-invariant, and hence $\cU'\subset
  \cU$.  On the other hand, if $y\in \cU$ then $T^j[y]\in \cU\subset
  \cP$ for all $j\in \Nset$.  Therefore, $\cU\subset \cU'$, also.
\end{proof}


We begin by collecting some basic results concerning polynomial
subspaces $\cU\subset \cP$.

\begin{definition}
  For a meromorphic function $f(z)$ we define
  $\ord_\zeta f,\; \zeta\in \Cset$ to be the largest integer $k$ such
  that $(z-\zeta)^{-k} f(z)$ is bounded as $z\to\zeta$; i.e., $k$ is
  the degree of the leading term in the Laurent expansion of $f$.  
  We will also employ the Landau $O$-notation to indicate local behaviour near
  $z=\zeta$. Thus,
  \[ f(z)\equiv g(z) +O((z- \zeta)^k),\; z\to \zeta \]
  means that $\ord_\zeta (f-g) \geq k,\; k\in \Zset$.  When no
  ambiguity arises, we omit the $z\to \zeta$.
\end{definition}

\begin{definition}\label{def:orddegseq}
  Let $\cU\subset \cP$ be a polynomial subspace. For a given
  $\zeta\in\Cset$, we define the \textit{order sequence of $T$ at
    $\zeta$} as
  \begin{equation}
    I_{\zeta} = \{ \ord_\zeta y: y \in \cU \}.
  \end{equation}
  We define $\nu_\zeta$ to be the cardinality of
  $\Nset\backslash I_\zeta$; that is, the number of gaps in the order
  sequence.
\end{definition}
\begin{prop}
  \label{prop:basis}
  Let $\cU\subset \cP$ be a polynomial subspace.  Then, for every
  $\zeta\in \Cset$, there exists a basis $\{y_k \}_{ k\in I_\zeta}$ of
  $\cU$ such that $\ord_\zeta y_k = k,\; k\in I_\zeta$.
\end{prop}
\begin{proof} Fix $\zeta\in \Cset$. For every $k\in I_\zeta$ choose a
  polynomial $y_k\in \cU$ such that $\ord_\zeta y_k = k$ and such that
  $\deg y_k$ is as small as possible.  
  We claim that
  $\{y_k\}_{k\in I_\zeta}$ is a basis of $\cU$.
  Suppose not. Set $\cU'= \lspan \{ y_k \} _{k\in I_\zeta}$ and let
  $f\in \cU\setminus \cU'$ be given.  Since the order of a polynomial
  cannot exceed its degree, we can choose a $g\in \cU'$ such that
  $\deg (f-g) - \ord_\zeta(f-g)$ is as small as possible. Let
  $k=\ord_\zeta (f-g)$.  Since $f-g\in \cU$ we must have
  $k\in I_\zeta$.  Hence, there exists a $c\in \Cset$ such that
  $\ord_\zeta(f-g-c y_k)>k$. By the way that $y_k$ was chosen,
  $\deg y_k \leq \deg(f-g)$ and hence
  \[\deg (f-g-cy_k) \leq \deg (f-g).\] We have thus,
  \[ \deg(f-g-c y_k) - \ord_\zeta(f-g-c y_k)< \deg(f-g) - \ord_\zeta(f-g),\] 
  which contradicts the assumption regarding $g$.
\end{proof}

\begin{prop}
  \label{prop:odgaps}
  Let $\cU\subset \cP$ be a polynomial subspace.  Suppose that the
  codimension $\nu = \dim \cP/\cU$ is finite.  Then, for every
  $\zeta\in \Cset$, we have $\nu_\zeta \leq \nu$.
\end{prop}
\begin{proof}
  Let
  \[y(z)=\sum_{n\notin I_\zeta} a_n (z-\zeta)^n ,\quad a_n\in \Cset.\]
  If $y\neq0$, then $ \ord_\zeta y$ is the smallest element of
  $\Nset\setminus I_\zeta$, which means that $y\notin \cU$.  Hence,
  the $\nu_\zeta$ polynomials $\{ (z-\zeta)^{n}\}_{n \notin I_\zeta}$
  are linearly independent modulo $\cU$.  By assumption, it is not
  possible to choose more than $\nu$ linearly independent polynomials
  in $\cP/\cU$ and the claim is thus established.
\end{proof}

\begin{prop}
  \label{prop:nbasis}
  Let $\cU\subset \cP$ be a finite-codimension polynomial subspace,
  $\zeta \in \Cset$ and $n=\max(\Nset\setminus I_\zeta)$, which is finite by the
  preceding Proposition.
  Then there exists a basis   $\{\ty_j\}_{j\in I_\zeta}$  of $\cU$ such that
  \begin{eqnarray}
    \label{eq:tyjj1}
    &\ord_\zeta \ty_j = j,\quad  \ty_j^{(j)} = 1,\quad & j\in I_\zeta,     \\
    \label{eq:tyji0}
    &\ty_j^{(i)}(\zeta) = 0, & i,j\in I_\zeta,\quad  j<i<n.
  \end{eqnarray}
\end{prop}
\begin{proof}
  By Proposition \ref{prop:basis}, there exists a basis
  $\{y_j\}_{j\in I_\zeta}$ of $\cU$ such that $\ord_\zeta y_j = j$.
  Let $m= \# \{ j\in I_\zeta \colon j<n \}$ and let  $j_1<j_2<\ldots <
  j_m$ be the elements of $\{ j\in I_\zeta \colon j<n \}$ arranged in
  ascending order.  
  Consider the $m\times m$ matrix $Y$ whose components are
  given by
  \[ Y_{ab} = y_{j_a}^{(j_b)}(\zeta) \]
  The assumption that $\ord y_j = j$, implies that $Y$ is upper
  triangular, with non-zero diagonal entries.  Hence, $Y$ is
  invertible. Let $C=Y^{-1}$ and set
  \begin{align*}
    \ty_{j_a}(z) &= \sum_{b=1}^m C_{ab}\, y_{j_b}(z),\;
                   C_{ab}\in \Cset,\quad a=1,\ldots, m,\\ 
    \ty_j(z) &= \frac{y_j(z)}{y_j^{(j)}(\zeta)},\quad  j>n,\; j\in
               I_\zeta. 
  \end{align*}
  Then,  \eqref{eq:tyjj1} and \eqref{eq:tyji0} hold by construction.
\end{proof}
\begin{definition}
  We define a differential functional with support at $\zeta\in \Cset$
  to be a linear map $\alpha:\cP\to \Cset$ of the form
  \[ \alpha[y] = \sum_{j=0}^k a_j y^{(j)}(\zeta),\quad a_j\in \Cset.\]
  We define the order of $\alpha$ to be the largest $j$ such that
  $a_j\neq 0$.  For $\cU\subset \cP$ and $\zeta\in \Cset$ we define
  $\Ann_\zeta\cU$ to be the vector space of differential functionals
  with support at $\zeta$ that annihilate $\cU$.
\end{definition}

\begin{prop}
  \label{prop:dfuncindep}
  A differential functional supported at $\zeta\in \Cset$ cannot be
  given as a finite linear combination of differential functionals with
  support at other points.
\end{prop}
\begin{proof}
  Let $\alpha_i,\; i=1,\ldots, m$ be differential functionals of order
  $k_i$ with support at $\zeta_i\in \Cset$.  Set
  \[ g(z) = \prod_{i=1}^{m} (z-\zeta_i)^{k_i+1}.\]
  Suppose that $\alpha$ is
  a differential functional of order $k$ supported at
  $\zeta\notin \{ \zeta_1,\ldots , \zeta_m\}$.  Let $L:\cP_k\to \cP_k$
  be the linear  transformation uniquely defined by the relation
  \[ \ord_\zeta(L(f) - gf) \geq k+1,\quad f\in \cP_k.\]
  Suppose that $L(f)=0,\; f\in \cP_k$.  Then,
  \[ g(z) f(z) = (z-\zeta)^{k+1} h(z),\quad h\in \cP.\]
  This is only possible if $f=h=0$. Hence, $\ker L$ is trivial and
   $L$ is invertible.  Since $\cP_k\not\subset\ker \alpha$, it is
  possible to choose an $f\in \cP_k$ such that $\alpha(L(f)) \neq 0$.
  Hence, $\alpha(fg)\neq 0$.  By construction,
  $\alpha_i[fg]=0,\; i=1,\ldots, m$.  Therefore $\alpha$ cannot be
  given as a linear combination of $\alpha_1,\ldots, \alpha_m$.
\end{proof}

\begin{prop}
   \label{prop:fc}
  For every $\zeta\in \Cset$ we have
  $\dim \Ann_\zeta\cU = \nu_\zeta$.
\end{prop}
\begin{proof}
  Suppose that $\nu_\zeta>0,\; \zeta\in \Cset$.  By Proposition
  \ref{prop:nbasis} there exists a basis $\{ \ty_j \}_{j\in I_\zeta}$
  of $\cU$ such that \eqref{eq:tyjj1} and \eqref{eq:tyji0} hold.
  Also, since $\ord_\zeta \ty_j = j$, we must have
  \[ \ty_j^{(i)}(\zeta)=0,\quad   i<j.\]
  For $k\notin I_\zeta$, set
  \begin{equation}
    \label{eq:alpha}
    \alpha_k[f] = f^{(k)}(\zeta)-\sum_{i<k \atop i\in I_\zeta}
    \ty_i^{(k)}(\zeta)\, f^{(i)}(\zeta),\quad f\in \cP. 
  \end{equation} 
  We claim that $\alpha_k\in \Ann_\zeta\cU$.  If $j>n$, then
  $\alpha_k[\ty_j] = 0$ because $\ord \alpha_k < j$.  If $j< n$, then
  \[ \alpha_k[\ty_j] = \ty_j^{(k)}(\zeta)-\sum_{i<k \atop i\in
    I_\zeta} \ty_i^{(k)}(\zeta)\, \ty_j^{(i)}(\zeta) =
  \ty_j^{(k)}(\zeta)- \ty_j^{(k)}(\zeta)\ty_j^{(j)}(\zeta) = 0.
  \]
  Next, we claim that the $\{ \alpha_k \}_{k\notin I_\zeta}$ are a
  basis of $\Ann_\zeta \cU$.  For $j\in I_\zeta,\;j<n$, let
  $p_j\in \cP_{n},\; $ be the $nth$ Taylor polynomial of $\ty_j(z)$
  around $z=\zeta$; i.e.,
  \[ \ty_j(z) \equiv p_j(z) + O((z-\zeta)^{n+1}),\; z\to \zeta.\]
  Let $\cU_{\zeta,n}\subset \cP_n$ be the span of these $p_j$.  Let
  $\talpha_j=\alpha_j | \cP_n$ denote the indicated restriction to
  $\cP_n$.  Observe that the $\alpha_k,\; k\notin I_\zeta$ have
  distinct orders, and that all such orders are $\leq n$.  Hence
  $\talpha_k,\; k\notin I_\zeta$ are also linearly independent.  Let
  $\cU_{\zeta,n}^\perp$ denote the vector space of linear forms on $\cP_n$
  that annihilate $\cU_{\zeta,n}$. Since $\cU_{\zeta,n}$ has codimension $\nu_\zeta$
  in $\cP_n$, we conclude that $\{ \talpha_k\}_{k\notin I_\zeta}$ is a
  basis of $\cU_{\zeta,n}^\perp$. Let $\alpha\in \Ann_\zeta\cU$ be given, and
  let $\talpha=\alpha|\cP_n$ be the indicated restriction.  Observe
  that $\ord \alpha\leq n$ because $\alpha[\ty_j]=0$ for all $j> n$.
  Hence, $\alpha$ is fully determined by $\talpha$.  Therefore,
  $\alpha$ belongs to the span of the $\alpha_k,\; k\notin I_\zeta$.
\end{proof}

\begin{definition}\label{def:poleT}
Let $T\in \Diff_2(\cQ)$ be an arbitrary exceptional
operator. We say that $\zeta\in\Cset$ is \textit{a pole of $T$ }if it is a pole of any of its coefficients $p,q,r\in \cQ$ as per \eqref{eq:Tgeneral}.  
\end{definition}

Let $\zeta_1,\dots,\zeta_N\in \Cset$ be the poles of $T$, and let
$\nu_i = \nu_{\zeta_i},\; i=1,\ldots, N$ be the number of gaps in each
order sequence, i.e. $\nu_i= \# \Nset\backslash I_{\zeta_i} $. We are
now ready to state the first main result
 
\begin{thm}
  \label{thm:codimsum}
 Let $T$ be an exceptional operator, $\cU$ its maximal polynomial invariant subspace and $\nu$ the codimension of $\cU$ in $\cP$. We have $\nu_\zeta>0,\; \zeta\in \Cset$ if and only if $\zeta$ is a
  pole of $T$.  Moreover,
  \begin{equation}
    \label{eq:codimsum}
    \nu = \sum_{i=1}^N \nu_i.
  \end{equation}
\end{thm}

\begin{proof}
  For $j\in
  \Nset$, set
  \[d_{ij} = \min \{ \ord_{\zeta_i} T^j[y] \colon y\in \cP \}.\]
  Consider the Laurent expansion of $T^j[y]$ at $z=\zeta_i$, and
  define differential functionals $\alpha_{kij}$ with
  support at $\zeta_i$ by means of the relation
  \begin{equation}
    \label{eq:Tjalpha}
    T^j[y] \equiv \sum_{k=1}^{-d_{ij}}
    \alpha_{kij}[y] (z-\zeta_i)^{-k} +  O(1),\quad z\to \zeta_i.
  \end{equation}
  By Proposition \ref{prop:Udef}, $\cU$ is the joint kernel of the
  $\alpha_{kij}$ defined above.  Since the codimension is finite, the
  joint kernel may be restricted to a finite number of triples
  $(i,j,k)$.  Hence
  $\Ann_\zeta \cU\subset \lspan \{ \alpha_{kij}\}_{i,j,k}$ for all
  $\zeta\in \Cset$. By Proposition \ref{prop:fc} if $\nu_\zeta>0$,
  then $\Ann_\zeta\cU$ is non-trivial, and hence $\zeta$ must be a
  pole of $T$.  Conversely, for every pole $\zeta_i\in \Cset$ is,
  there is at least one $\alpha_{kij}$ that annihilates $\cU$.  Hence
  $\nu_i>0$ for all $i$.  Therefore $\nu_\zeta>0$ if and only if
  $\zeta\in\Cset$ is a pole of $T$.  It also follows that $\cU$ is the
  joint kernel of $\oplus_{i=1}^N \Ann_{\zeta_i}\cU$.  Relation
  \eqref{eq:codimsum} now follows by Proposition
  \ref{prop:dfuncindep}.
\end{proof}

\section{The structure theorem for exceptional  operators} 
\label{sec:struct}


Let $T$ be an exceptional operator with eigenpolynomials
$y_k \in \cP_k^*$.  Observe that for every
$\sigma\in \cP_n^*,\; n\geq 1$, the gauge-equivalent operator
$\tT = \sigma T \sigma^{-1}$ is also exceptional with eigenpolynomials
$\ty_{k+n} = \sigma y_{k}\in \cP_{k+n}$.  Thus every exceptional
operator is gauge equivalent to infinitely many other exceptional
operators. However, as we show below, every gauge-equivalent class of
exceptional operators admit a distinguished gauge, as per the
following.

\begin{definition} \label{def:natural} We will say that an operator
  $T\in \Diff_2(\cQ)$ with coefficients $p,q,r$ as per
  \eqref{eq:Tgeneral} is a \emph{natural} operator if $p\in \cP_2$ and
  if there exist polynomials $s\in \cP_1$ and $\eta\in \cP$ such that
  $T[y]=0$, when multiplied by $\eta$, is equivalent to the bilinear
  relation
  \begin{equation}
    \label{eq:bilinear}
    p(\eta y''  - 2 \eta' y'+\eta'' y ) + \frac{1}{2} p' ( \eta
    y' + \eta' y) + s (\eta y'-\eta' y) =0.
  \end{equation}
\end{definition}
\begin{remark}
  An equivalent formulation of the above definition is that $T$ is a natural operator if $p\in\cP_2$ and the other two coefficients have the form
  \begin{subequations}
    \label{eq:natqr}
  \begin{align}
    \label{eq:natqform}
    q &= \frac{p'}{2} + s - \frac{2p\eta'}{\eta}\\
    \label{eq:natrform}
    r &=\frac{p\eta''}{\eta} + \left(\frac{p'}{2} -s\right)
    \frac{\eta'}{\eta}.
  \end{align}
  \end{subequations}
  for some polynomials $s\in \cP_1$ and $\eta\in \cP$.
\end{remark}

The main result in this Section is a structure theorem for the coefficients of an exceptional operator $T$.

\begin{thm}
  \label{thm:natural}
  Let $T=p(z)D_{zz} + q(z) D_z + r(z)$ be an exceptional operator.
  Then, $p\in \cP_2$ while $q$ has the form shown in
  \eqref{eq:natqform} for some $s\in \cP_1$ and
  \begin{equation}
    \label{eq:etadef}
    \eta(z) = \prod_{i=1}^N (z-\zeta_i)^{\nu_i}
  \end{equation}
  where $\zeta_i,\, i=1\,\ldots, N$ are the poles of $T$, and
  $\nu_i=\nu_{\zeta_i}$ are the corresponding gap cardinalities.
  Moreover, $T$ is gauge equivalent to a natural operator; i.e. modulo
  a gauge-transformation $r$ has the form shown in
  \eqref{eq:natrform}.
\end{thm}
\begin{remark}
  Note that while the above theorem states that an exceptional
  operator must have a very specific form, the final characterization
  of an exceptional $T$ is even more restrictive. Indeed, the poles
  $\zeta_i$ of an exceptional operator cannot be chosen at will, but
  will need to satisfy a set of constraints that, in similar contexts,
  have been called the \textit{locus equations}
  \cite{airault1977,chalykh1999}. Equivalently, every exceptional
  operator $T$ is gauge equivalent to a natural operator, but not
  every natural operator is exceptional.
\end{remark}
\noindent
We devote the rest of this section  to the proof of this theorem.

It turns out that every equivalence class of gauge-equivalent
exceptional operators admits another distinguished gauge, as per the
following.
\begin{definition}\label{def:reduced}
  Let $T\in \Diff_2(\cQ)$ be an exceptional operator and
  $\cU\subset \cP$ the corresponding maximal invariant polynomial
  subspace.  We will say that $T$ is \emph{reduced} if there
  \emph{does not exist} a $\zeta\in \Cset$ such that $y(\zeta)=0$ for
  all $y\in \cU$.
\end{definition}

\begin{prop}
  \label{prop:greduced}
  Every exceptional operator is gauge-equivalent to a reduced
  exceptional operator.
\end{prop}
\begin{proof}
  Suppose that $\tT\in \Diff_2(\cQ)$ is exceptional with
  eigenpolynomials $\ty_k\in \cP^*_k$. Let $\sigma\in \cP$ be the
  polynomial GCD of the $\ty_k$. Then, the
  operator \[ T = \sigma^{-1} \tT \sigma ,\] admits polynomial
  eigenfunctions $\sigma^{-1} \ty_k\in \cP^*_{k-\deg\sigma}$ which, by
  construction, do not possess a common root.
\end{proof}

\begin{example}
  Unreduced operators are, for all practical purposes, equivalent to
  their reduced counterparts.  For example, consider the classical
  Hermite differential equation
  \[ y''-2zy'+2ny = 0,\quad n=0,1,2,\ldots \] whose polynomial
  solutions are the classical Hermite polynomials $y=H_n(z)$. One
  could instead consider the polynomials $\hat{H}_n(z) = (1+z^2)
  H_{n-2}(z),\; n\geq 2$.  By construction, $y=\hat{H}_n$ is a
  solution of the differential equation
  \[ y''- 2\left(z + \frac{2z}{1+z^2}\right)y'+
  \left(4+2n+\frac{2}{1+z^2} -\frac{8}{(1+z^2)^2}\right) y = 0,\]
  which is obtained by conjugating the classical Hermite operator by
  the multiplication operator $1+z^2$. The ordinary Hermite
  polynomials are orthogonal on $(-\infty,\infty)$ relative to the
  weight $e^{-z^2}$, and hence by construction the modified
  polynomials $\hat{H}_n(z)$ are orthogonal relative to the weight
  $e^{-z^2}/(1+z^2)^2$.  Thus, $\hat{H}_n,\; n\geq 2$ constitute a family
  of exceptional orthogonal polynomials with 2 missing degrees. This
  type of construction is quite general, but does not produce
  genuinely new orthogonal polynomials.
\end{example}

\begin{remark}
  The reduced gauge is not necessarily unique.  As an example,
  consider the classical Laguerre operator
  \begin{equation}
    \label{eq:cLaguerre}
    \cL_{\alpha} = z D_{zz} + (1+\alpha-z) D_z. 
  \end{equation}
  The Laguerre polynomials 
  \[ L^{(\alpha)}_n(z) = \frac{z^{-\alpha} e^z}{n!} D_z^n[
  e^{-z}z^{n+\alpha}],\quad n\in \Nset \] are polynomial
  eigenfunctions with 
  \[ \cL_{\alpha}[ L^{(\alpha)}_n] = -n L^{(\alpha)}_n.\]
  Since $L^{(\alpha)}_0=1$ is a constant, $\cL_{\alpha}$ is reduced
  for all $\alpha$.
  Now suppose that $\alpha=m>0$ is a positive integer. A direct
  calculation shows that
  \[ z^m \cL_{m} z^{-m} = \cL_{-m} +m .\]
  Both $\cL_{m}$ and $\cL_{-m}$ are reduced, exceptional operators
  but they are related by a non-trivial gauge transformation.  At the
  root of this non-uniqueness is the fact that $\cL_{m}$ possesses
  rational, non-polynomial eigenfunctions.  Indeed by
  \cite{Szego1939}[Section 5.2],
  \[ L^{(-m)}_n(z) = (-z)^m \frac{(m-n)}{n!} L^{(m)}_{n-m}(z),\quad
  n\geq m.\]
  Therefore, $z^{-m} L^{(-m)}_n(z)$ is a rational eigenfunction of
  $\cL_{m}$.  In the absence of such rational, nonpolynomial
  eigenfunctions it seems reasonable to conjecture that the reduced
  condition fixes a unique gauge, but we will not pursue this question
  here.  For us the reduced gauge is an important, but technical
  condition that simplifies some of the arguments in the proof of
  Theorem \ref{thm:natural}.

  Before moving on, we note that there may even be an infinite number
  of distinct, but gauge-equivalent reduced exceptional operators.
  Consider, for example, Euler operators, that is operators of the
  form
  \[ T=a z^2 D_{zz} + b z D_z + c,\quad a,b,c\in \Cset.\]
  For such operators, every monomial $z^k,\; k\in \Zset$ is an
  eigenfunction. Thus, every Euler operator $T$ is reduced, but so is
  $z^n T z^{-n}$ for every $n\in \Zset$.  Also note that all Euler
  operators are in natural form, so that uniqueness also fails in
  Theorem \ref{thm:natural}.
\end{remark}


We see that every class of gauge-equivalent exceptional operators has
at least two distinguished gauges: the natural gauge (Defintion
\ref{def:natural}) and the reduced gauge (Defintion
\ref{def:reduced}).  Usually, these two choices of gauge are the same,
but this is not always the case, as illustrated by the example below.
Lemma \ref{lem:redform}, proved below, shows that the natural and
reduced gauges are not the same precisely when the denominator
polynomial $\eta(z)$ of the exceptional weight has repeated roots
\cite{sasaki2012global,ho2013confluence}.

\begin{example}
  \label{ex:rednatdiff}
  The following example illustrates the difference between the natural
  and reduced gauge of an exceptional operator.  The example is based
  on the following family of two-step exceptional Laguerre
  polynomials \cite{Gomez-Ullate2012}.  Let $L^{(\alpha)}_n(z)$ denote the classical Laguerre
  polynomial of degree $n$. For $n\geq 2$ set
  \begin{equation}
    \hL^{(\alpha)}_n (z):= e^{-z} \Wr\big[L^{(\alpha)}_{n-2}(z),L^{(\alpha)}_1(z), e^z
    L^{(\alpha)}_2(-z)\big].
  \end{equation}
  By
  construction, $\hL_3^{(\alpha)}(z)=0$, and so we obtain a
  codimension-3 family of polynomials with degrees $n=2,4,5,6,\ldots$.
  These polynomials can also be given using the following form
  introduced by Dur\'an \cite{Duran2014a}
  \begin{equation}
    \hL^{(\alpha)}_n (z)=
    \begin{vmatrix}
      L^{(\alpha)}_{n-2}(z) & -L^{(\alpha+1)}_{n-3}(z) & L^{(\alpha+2)}_{n-4}(z) \\
      L^{(\alpha)}_1(z)    &  -L^{(\alpha+1)}_0(z)  & 0\\
      L^{(\alpha)}_2(-z)&     L^{(\alpha+1)}_2(-z)&     L^{(\alpha+2)}_2(-z)
    \end{vmatrix},\quad n=2,4,5,6,\ldots
  \end{equation}
  where $L^{(\alpha)}_j(z)$ is understood to be zero for $j<0$.  

  Let
  \begin{align*}
    \eta^{(\alpha)}(z) &= e^{-z} \Wr\big[L^{(\alpha)}_1(z), e^z
                         L^{(\alpha)}_2(-z)\big]\\
                       &=  \begin{vmatrix}
                         L^{(\alpha)}_1(z) & -1 \\
                         L^{(\alpha)}_2(-z) & L^{(\alpha+1)}_2(-z)
                       \end{vmatrix} \\
                       &=-\frac12 \big(z^3 + (\alpha+4)z^2-(\alpha+4)(\alpha+1)z -
                         (\alpha+1)(\alpha+2)(\alpha+4)\big).
  \end{align*} The polynomial family $\hL^{(\alpha)}_n(z),\;
  n=2,4,5,\ldots$ is exceptional and in the natural gauge, because of
  the following bilinear relations:
  \begin{align}
    \label{eq:ngde}
    & z \left( \eta^{(\alpha)}\hL_n^{(\alpha)}{}'' - 2
      \eta^{(\alpha)}{}' \hL_n^{(\alpha)}{}'+\eta^{(\alpha)}{}''
      \hL_n^{(\alpha)}\right) + \frac12 \left(\eta^{(\alpha)}
      \hL^{(\alpha)}_n{}' +
      \eta^{(\alpha)}{}' \hL_n^{(\alpha)}\right)\\ \nonumber
    &\qquad + \left(-z+\alpha+\frac52\right)\left(\eta^{(\alpha)}
      \hL^{(\alpha)}_n{}' - \eta^{(\alpha)}{}' \hL_n^{(\alpha)}\right)
      + (n-3) \hL_n^{(\alpha)} \eta^{(\alpha)} = 0
  \end{align}

  It is easy to check that $\eta^{(\alpha)}(z)\neq 0$ for $z\in
  [0,\infty)$ if and only if $\alpha\in (-\infty,-4)\cup (-2,-1)$.
  Hence, for $\alpha\in (-2,-1)$ the polynomials $\hL^{(\alpha)}_n(z)$
  are orthogonal with respect to the inner product
  \[  \left< f,g\right> = \int_0^\infty \frac{z^{\alpha+2}
    e^{-z}}{\big(\eta^{(\alpha)}(z)\big)^2} f(z) g(z) dz.\]

  The discriminant of $\eta^{(\alpha)}(z)$ is $\frac18
  (\alpha+1)(\alpha+4)^2(4\alpha+7)^2$.  Hence, for $\alpha=-\frac74$
  the denominator polynomial has a multiple root.  Indeed,
  \[ \eta^{\lp-\frac74\rp}(z) = -\frac12 \lp z+\frac34\rp^3;\] there
  is a single root with a triple multiplicity.  Moreover,
  \begin{align*}
    L^{\lp-\frac74\rp}_2(-z) &= \frac12\left( z+\frac34\right)\lp z-\frac14\rp\\
    L^{\lp-\frac74\rp}_1(z) &= -\lp z+\frac34\rp.
  \end{align*}
  Hence,
  \begin{align*}
    \hL^{\lp-\frac74\rp}_n(z) &= -e^{-z} \lp z+\frac34\rp^3 \Wr\left[
                                \frac{L^{\lp-\frac74\rp}_{n-2}(z)}{z+\frac34} , 1, \frac12 e^z
                                \lp z-\frac14\rp\right]\\
                              &= -\frac12\lp z+\frac34\rp^3 L^{\lp \frac14\rp}_{n-4}(z)     
                                -\frac12 \lp z+\frac34\rp^2 \lp z+\frac{15}4\rp
                                L^{\lp -\frac34\rp}_{n-3}(z) - \frac12  
                                \lp z+\frac34\rp\lp z+\frac{15}4\rp L^{\lp-\frac74\rp}_{n-2}(z)
  \end{align*}
  has a root at $z=-\frac34$ for every $n$.  Thus, for $\alpha=-\frac74$
  the natural gauge does not agree with reduced gauge.

  Let us therefore introduce the reduced family of polynomials
  \[ \tL_n(z) = \lp
  z+\frac34\rp^{-1}\hL^{\lp-\frac74\rp}_{n+1}(z),\quad
  n=1,3,4,\ldots.\] This family of polynomials is exceptional and
  reduced.  The reduced inner product is
  \[ \left< f,g\right> = \int_0^\infty \frac{z^{\frac14} e^{-z}}{\lp
    z+\frac34\rp^4} f(z) g(z) dz.\]
  To obtain the corresponding differential equation we conjugate
  \eqref{eq:ngde} by $z+\frac34$.  Applying the gauge-transformation
  law \eqref{eq:pqrWgaugelaw}, we obtain the differential equation
  \[ z \tL_n'' + \left(\frac54-z\right) \tL_n' + (n-1) \tL_n- \frac{4z
    \tL_n' + \tL_n}{z+\frac34}=0.\] In this way we recover the
  codimension 2 exceptional family first described in \cite[Section
  6.2.5]{Gomez-Ullate2012}.  This example also serves as an
  illustration of the principle that codimension very much depends on
  the choice of gauge.  The generic family described above has
  codimension 3.   However, for one particular value of the parameter,
  the ``true'' codimension, that is the codimension of the
  corresponding reduced family, is actually 2.
\end{example}

We begin with some Lemmas.  Below $p,q,r$ are the coefficients of
$T\in \Diff_2(\cQ)$ as per \eqref{eq:Tgeneral}.

\begin{definition}
  \label{def:Tlaurent}
  We define the \textit{Laurent decomposition} of $T$ at a given
  $\zeta\in \Cset$ to be the sum
  \begin{equation}
    \label{eq:TLaurent}
    T = \sum_{j\geq d_\zeta} T_j,
  \end{equation}
  where 
  \begin{align}
    \label{eq:Tj}
    T_{j} &= p_{j+2}(z-\zeta)^{j+2} D_{zz}+ q_{j+1} (z-\zeta)^{j+1}
            D_z +r_j (z-\zeta)^{j},\intertext{with } \nonumber p(z) &=
                                                                      \sum_{j\geq \ord_\zeta p} p_j(z-\zeta)^j,\quad p_j\in \Cset,\\
    \nonumber q(z) &= \sum_{j\geq \ord_\zeta\! q} q_j
                     (z-\zeta)^j,\quad q_j\in \Cset,\\  
    \nonumber r(z) &= \sum_{j\geq \ord_\zeta\!  r} r_j
                     (z-\zeta)^j,\quad r_j\in \Cset
  \end{align}
  the Laurent decompositions of $p,q,r$, respectively.  The \emph{leading
    order} of the expansion is the integer $d_\zeta$ given by
  \begin{equation}
    \label{eq:ordT}
    d_\zeta  =  \min\{\ord_\zeta p - 2,\,\ord_\zeta \!q-1
    ,\,\ord_\zeta\! r\}.
  \end{equation}
\end{definition}

\begin{lem}
  \label{lem:skeleton}
  If $T$ is exceptional, then $T_{d_\zeta},\; \zeta\in \Cset$
  preserves $\lspan \{ (z-\zeta)^k \colon k\in I_\zeta\}$.
\end{lem}
\begin{proof}
  Let $\cU\subset \cP$ be the maximal invariant polynomial subspace as
  per \eqref{eq:Udef}. By Proposition \ref{prop:basis}, there exists a
  basis of $\cU$ of the form
  \[ y_k(z) \equiv (z-\zeta)^k + O((z-\zeta)^{k+1}),\; z\to
  \zeta,\qquad k\in I_\zeta.\]
  Since $\cU$ is $T$ invariant and $T_{d_\zeta}$ is the smallest order
  term of $T$, the desired conclusion follows.
\end{proof}

%


\begin{lem}
  \label{lem:gapstring}
  If $T\in \Diff_2(\cQ)$ is exceptional and
  $d_\zeta <0,\; \zeta\in \Cset$, then for every natural number
  $j\notin I_\zeta$, there exists a natural number $n_j> 0$ such that
  \begin{enumerate}
  \item[i)] $j,j-d_\zeta,\dots,j-(n_j-1)d_\zeta\notin I_\zeta$;
  \item[ii)] $j-d_\zeta n_j\in I_\zeta$ and
    $T_{d_\zeta}[(z-\zeta)^{j-d_\zeta n_j}] = 0$.
  \end{enumerate}  
\end{lem}
\begin{proof}
  If $j\notin I_\zeta$ then by Lemma \ref{lem:skeleton} either
  $j-d_\zeta\notin I_\zeta$ or
  $T_{d_\zeta}\left[(z-\zeta)^{j-d_\zeta}\right]=0$. Iterating this
  argument, and using Proposition \ref{prop:odgaps} and the fact that the
  codimension is finite, we see that the first possibility can happen
  only a finite number of times.
\end{proof}

\begin{lem}
  \label{lem:dleq2}
  If $T$ is exceptional, then $d_\zeta\geq -2$ for every
  $\zeta\in\Cset$.
\end{lem}
\begin{proof}
  Suppose that $d_\zeta<-2$. For each $j\in\{0,1,2\}$, if $j\in
  I_\zeta$ then $T_{d_\zeta}[(z-\zeta)^j]=0$. If $j\notin I_\zeta$,
  then by Lemma \ref{lem:gapstring} there exists an integer $n_j>0$
  such that $T_{d_\zeta}\left[(z-\zeta)^{j-n_j d_\zeta}\right]=0$. In
  all cases, we see that $T_{d_\zeta}$ would be required to annihilate
  $(z-\zeta)^k$ for three different integers $k$, and since it is a
  second order operator, this is impossible.
\end{proof}

\begin{lem} \label{lem:qpoles} If $T$ is exceptional, then $p(z)$ is a
  polynomial, and the poles of $q(z)$ are simple.
\end{lem}
\begin{proof}
  If $\zeta\in\Cset$ is a pole of $p(z)$, then by \eqref{eq:ordT} we
  would have $d_\zeta\leq-3$ which is forbidden by Lemma
  \ref{lem:dleq2}.    To prove the second claim, 
  note that if $\ord_\zeta q < -1$, then $d_\zeta\leq -3$, which is
  again forbidden by Lemma \ref{lem:dleq2}.
\end{proof}

We now prove a number of structural Lemmas about reduced exceptional
operators.  Proposition \ref{prop:greduced} allows us to extend these
results to exceptional operators that are not necessarily reduced.

\begin{lem}
  \label{lem:ordergaps}
  If  $T$ is reduced and $\nu_\zeta>0,\; \zeta\in \Cset$, then 
  \begin{equation}
    \label{eq:Ibform}
    I_\zeta = \{ 2j \colon j\in \Nset,\;  j \leq \nu_\zeta\} \cup
    \{ n\in \Nset \colon n    \geq     2\nu_\zeta+1 \}.
  \end{equation}
  Moreover, $p(\zeta)\neq 0$, with
  \begin{equation}\label{eq:T-2}
    T_{-2} = p(\zeta) \left(D_{zz} - \frac{2\nu_\zeta}{(z-\zeta)} D_z\right).
  \end{equation}
\end{lem}
\begin{proof} 
  By Theorem \ref{thm:codimsum}, $\zeta$ is a pole of $T$, and hence
  $d_\zeta<0$.    As per \eqref{eq:Tj}, write
  \begin{align*}
    T_{-2} &= p_0 D_{zz} + q_{-1} (z-\zeta)^{-1} D_z+
             r_{-2}(z-\zeta)^{-2},\\
    T_{-1} &= p_1 (z-\zeta) D_{zz} + q_{0} D_z+ r_{-1}(z-\zeta)^{-1},
  \end{align*}
  Since $T$ is reduced, $0\in I_\zeta$.  Hence, by Lemma
  \ref{lem:skeleton}, $T_{d_\zeta}[1]=0$.  Hence, $d_\zeta=-2$,
  because otherwise $p_0=q_{-1}=r_{-2} = r_{-1}=0$, which violates the
  assumption that $\zeta$ is a pole.  By Lemma \ref{lem:gapstring},
  $T_{-2}[(z-\zeta)^k]=0$ for some $k\geq 2$.  Since $T_{-2}$ cannot
  annihilate 3 different powers, $1\notin I_{\zeta}$.  Hence, by Lemma
  \ref{lem:gapstring}, there exists an $n\geq 1$ such that
  $1,3,5,\dots,2n-1\notin I_\zeta$, and
  \[ T_{-2}[(z-\zeta)^{2n+1}] =0.\]
  Since $T_{-2}$ annihilates $1$ and $(z-\zeta)^{1+2n}$, it cannot
  annihilate another monomial, which proves \eqref{eq:Ibform}.  By
  Lemma \ref{lem:qpoles}, \eqref{eq:T-2} must hold with
  $\nu_\zeta = n$.
\end{proof}

\begin{lem}
  \label{lem:redpoles}
  Suppose that $T$ is exceptional and reduced. Then,
  \begin{itemize}
  \item[(i)] the poles of $q(z)$ are distinct from the zeros of
    $p(z)$;
  \item[(ii)] the poles of $r(z)$ are also the poles of $q(z)$;
  \item[(iii)] the poles of $r(z)$ are simple.
  \end{itemize}
\end{lem}
\begin{proof}
  Claim (i) follows from \eqref{eq:T-2}.
  Since $T$ is reduced, there exists a $y_0\in \cU$ satisfying
  \[ y_0(z)\equiv 1 +a (z-\zeta)  + O((z-\zeta)^2),\quad z\to\zeta.\]
  Suppose that $z=\zeta$ is a pole of $r(z)$. Employing the notation
  of the proof of Lemma \ref{lem:ordergaps}, we must have $r_{-2}=0$
  and
  \[ T[y_0] \equiv (a q_{-1} + r_{-1})(z-\zeta)^{-1} +
  O(1),\quad z\to \zeta.  \]
  Hence, $r_{-1} =- a q_{-1}$, which implies that $q_{-1} \neq 0$.
  This proves (ii) and (iii).
\end{proof}

Recall that $z=\zeta$ is an \emph{ordinary}
point of the differential equation
\[ y''(z) + \frac{q(z)}{p(z)} y'(z) + \frac{r(z)}{p(z)} y(z)=0,\]
if $q/p$ and $r/p$ are analytic at $z=\zeta$.  If the above quotients
are singular, but if
\begin{equation}
  \label{eq:ordpqr}
  \ord_\zeta \left(\frac{q}{p}\right) \geq -1,\quad \ord_\zeta
  \left(\frac{r}{p}\right)\geq -2,  
\end{equation}
then $z=\zeta$ is called a \emph{regular singular} point of $T$.  Also
recall that the above differential equation admits two linearly
independent series solutions, in the sense of the method of Frobenius,
if and only if $z=\zeta$ is either an ordinary point or a regular
singular point.  For more details, see \cite{Ince2003}[Section 15.3,
Section 16.1-16.3].  Finally, observe that in light of
\eqref{eq:ordT}, condition \eqref{eq:ordpqr} can be restated more
simply as
\begin{equation}
  \label{eq:dzetap}
  d_\zeta = \ord_\zeta p-2.
\end{equation}

By Lemmas \ref{lem:qpoles} and \ref{lem:redpoles} every
$\zeta\in \Cset$ is either an ordinary point or a regular singular
point of a reduced operator $T$.  By \eqref{eq:pqrWgaugelaw}, the same
is true for a general exceptional operator. We therefore introduce the
following terminology.
\begin{definition}
  We say that $z=\zeta$ is
  \begin{itemize}
  \item[i)] a \emph{primary pole} if it is a pole of $q(z)$ or $r(z)$;
  \item[ii)] a \emph{secondary pole} if it is not a primary pole, but it is
    a zero of $p(z)$;
  \item[iii)] an \textit{ordinary point}  otherwise.
  \end{itemize} 
\end{definition}

\begin{remark}
  By Lemma \ref{lem:redpoles}, if $T$ is reduced, then primary poles
  are the same as the poles of $q(z)$.  As the following example
  shows, this need not be the case for unreduced exceptional
  operators.
\end{remark}

\begin{example}
  Let $m>0$ be a positive integer and consider the conjugation of the
  classical Laguerre operator \eqref{eq:cLaguerre},
  \[ T = z^m\circ \cL_{\alpha}\circ z^{-m}+m = z D_{zz} +
  (\alpha+1-2m-z)D_z + \frac{m-\alpha}{z} .\]
  By construction, this is an exceptional, albeit unreduced, operator
  with gaps in degrees $n=0,1,\ldots, m-1$. The unique pole is at
  $z=0$, which also happens to be a zero of $p(z)=z$.  Also note that
  in this case, $z=0$ is a pole of the operator, but not a pole of
  $q(z)$.
\end{example}

We now recall some key notions relating to logarithmic singularities
from the point of view of Frobenius' method.

\begin{definition}
  We say that $T\in \Diff_2(\cQ)$ has \textit{trivial monodromy} at
  $\zeta\in \Cset$ if $T[y]=0$ admits two linearly independent Laurent
  series solutions, i.e. if the general solution of $T[y]=0$ is
  meromorphic in a neighbourhood of $\zeta$.
\end{definition}

If $T$ is reduced, then at a primary pole, it can be seen from
\eqref{eq:T-2} that the two roots of the indicial equation are $0$ and
$2\nu_\zeta+1$. Since they differ by an integer, there is the
possibility that one of the solutions has a logarithmic singularity.
We now show that the assumption that $T$ is exceptional precludes that
possibility.

\begin{prop}
  \label{prop:trivmonod}
  Let $T=p(z) D_{zz} +q(z) D_z + r(z)$ be an exceptional operator.  If
  $p(\zeta)\neq 0,\;\zeta\in \Cset$, then $T$ has trivial monodromy at
  $z=\zeta$.
\end{prop}
\begin{proof}
  Without loss of generality $p(\zeta)=1$.  By Proposition
  \ref{prop:greduced}, there is no loss of generality, if we suppose
  that $T$ is reduced.  If $z=\zeta$ is not a pole of $q(z)$, then
  by Lemma \ref{lem:redpoles} it is an ordinary point, in which case
  $T[y]=0$ admits two independent power series solutions around
  $z=\zeta$.  We therefore assume that $z=\zeta$ is a pole of $q(z)$,
  and hence that $\nu_\zeta>0$.  By Lemma \ref{lem:dleq2}, and by the
  assumption on $p(\zeta)$ we have $d_\zeta = -2$.  Indeed, by Lemma
  \ref{lem:ordergaps},
  \begin{equation}
   \label{eq:T-2form}
    T_{-2} = D_{zz} - \frac{2\nu_\zeta}{z-\zeta} D_z.    
  \end{equation}

  Use Proposition \ref{prop:basis} to choose a basis
  $\{  y_j \}_{j\in I_\zeta}$ of $\cU$ such that  $\ord_\zeta y_j =
  j$.  Without loss of generality, 
  \[ y_j(z) = (z-\zeta)^j + O((z-\zeta)^{j+1}),\quad z\to \zeta.\]
  Hence, a formal series
  \[ a(z)=\sum_{i\in I_\zeta} a_i\, y_i(z),\; a_i\in \Cset \]
  defines a power series around $z=\zeta$, with the coefficient of
  $(z-\zeta)^k,\; k\in \Nset$ being a finite linear combination of the
  $a_i,\; i\in I_\zeta$ such that $i\leq k$.  Since $\cU$ is
  $T$-invariant and $d_\zeta =-2$, for a given $i\in I_\zeta$ we have
  \[ T[y_i] = \sum_{j\geq i-2 \atop j\in I_\zeta} B_{ij} \,y_{j} ,\;
  B_{ij} \in \Cset,\]
  with $B_{ij}=0$ for $j$ sufficiently large. 
  Thus, $T[a]=0$ if and only if
  \[ \sum_{i\leq j+2 \atop i\in I_\zeta} a_i B_{ij} = 0\]
  for all $j\in I_\zeta$.
  By
  \eqref{eq:T-2form},
  \[ B_{i,i-2} = i (i-1-2\nu_\zeta),\quad i\in I_\zeta.\]
  Thus, $T[a]=0$ if and only if
  \begin{equation}
    \label{eq:ajrecursive}
    (j+2)(j+1-2\nu_\zeta)a_{j+2} + \sum_{i\leq j+1\atop i\in
      I_\zeta}  B_{ij}\,a_i =0 
  \end{equation}
  for all $j\in I_\zeta$.  By Lemma \ref{lem:ordergaps},
  \[ I_\zeta = \{ 0,2,4,\ldots, 2\nu_\zeta-2,
  2\nu_\zeta,2\nu_\zeta+1,2\nu_\zeta+2,2\nu_\zeta+3,\ldots \}.\]
  Hence, 
  \[ (j+2)(j+1-2\nu_\zeta)\neq 0,\quad j\in I_\zeta,\]
  and relations
  \eqref{eq:ajrecursive} recursively define $a_j,\; j\in I_\zeta$ for arbitrary values of $a_0, a_{2\nu_\zeta+1}$.
  Since there are two linearly independent power series solutions of
  $T[y]=0$ at $z=\zeta$, the operator $T$ has trivial monodromy there.
\end{proof}
\begin{remark}
  If $T\in \Diff_2(\cQ)$ is exceptional, then so is $T-\lambda$ for
  every $\lambda\in \Cset$.  Hence, if $T$ is exceptional, then the
  general solution of the eigenvalue equation $T[y]=\lambda y$ is
  meromorphic away from secondary poles.
\end{remark}

Lemma \ref{lem:ordergaps} established the form of the $T_{-2}$ term of
a reduced, exceptional operator.  The conclusion is that the Laurent
expansion of $q(z)$ at $z=\zeta_i$ has the form
\begin{equation}
  \label{eq:qlaurent}
  q(z) \equiv \frac{-2\nu_i\, p_{i0}}{z-\zeta_i} + q_{i0}
  + O( (z-\zeta_i)),\quad z\to \zeta_i,\qquad q_{i0}\in \Cset
\end{equation}
so that
\[ T_{-2} = p_{i0} \left(D_{zz} - \frac{2\nu_i}{z-\zeta_i}
  D_z\right),\] where
\[ p_{i0} = p(\zeta_i) \neq 0.\] Using the trivial monodromy
results we can now describe the $T_{-1}$ term.
\begin{lem}
  \label{lem:T-1}
  If $T\in \Diff_2(\cQ)$ is reduced and $z=\zeta_i$ one of the primary
  poles, then
  \begin{equation}
    \label{eq:T-1}
    T_{-1} = p_{i1} \left( (z-\zeta_i) D_{zz} - \frac{      \frac12
        \nu_i(3\nu_i-1)}{z-\zeta_i} \right) + q_{i0}\left(D_z -  
      \frac{\nu_i}{z-\zeta_i}\right),
  \end{equation}
  where $p_{i1} = p'(\zeta_i)$.
\end{lem}

The proof is based on the following result characterizing
monodromy-free Schr\"odinger operators \cite[Proposition
3.3]{Duistermaat1986}.
\begin{lem}[Duistermaat-Gr\"unbaum]
  \label{lem:dg}
  Let $U(x)$ be meromorphic in a neighborhood of $x=0$ with Laurent
  expansion 
  \[ U(x) = \sum_{j\geq -2} c_j x^j.\] Then all eigenfunctions of the
  Schr\"odinger operator $H=-D_{xx} + U(x)$ are single-valued around
  $x=0$ if and only if
  \[ c_{-2} = \nu (\nu+1) \]
  for some integer $\nu\geq 1$, and 
  \[ c_{2j-1} = 0,\quad 0\leq j \leq \nu. \]
\end{lem}

\begin{proof}[Proof of Lemma \ref{lem:T-1}.]
  Since $p(\zeta_i)\neq 0$ we can find an analytic change of variables
  $z=\zeta(x)$ that satisfies
  \begin{equation}
    \label{eq:z'x}
    \zeta'(x)^2 = p(\zeta(x)),\quad \zeta(0) = \zeta_i.
  \end{equation}
  Explicitly,
  \[   x= \int^{z=\zeta(x)}\!\! \frac{dz}{\sqrt{p(z)}}.\]
  In this way 
  \[ D_{xx} = p(z) D_{zz} + \frac12 p'(z) D_z.\]
  Set
  \[ \mu(z) = \exp\left(\frac12 \int \frac{q(z) - \frac12 p'(z)}{p(z)}
    dz \right). \] Observe that $\mu(z)$ is analytic at $z=\zeta_i$.
  A direct calculation shows that
  \begin{align*}
    \mu T \mu^{-1} &= p(z) D_{zz} + \frac12 p'(z) D_z + V(z),
  \end{align*}
  where
  \[ V(z) = \frac{p''(z)}{4} - \frac{q'(z)}{2}- \frac{\big(q(z) - \frac12
    p'(z)\big) \big( q(z) - \frac32 p'(z)\big)}{4p(z)} + r(z).\] Set 
  \[ H = - D_{xx} - V\big(\zeta(x)\big), \] so that $T[y] = \lambda y$ if and
  only if $H[\psi] = -\lambda \psi$, where
  \[ \psi(x) = \mu\big(\zeta(x)\big) y\big(\zeta(x)\big).\] Hence, $T$ has trivial
  monodromy at $z=\zeta_i$ if and only if $H$ has trivial monodromy at
  $x=0$.  Using \eqref{eq:qlaurent} and a direct calculation, gives
  \begin{align*}
    V(z) \equiv-\frac{\nu_i(\nu_i+1) p_{i0}}{(z-\zeta_i)^2} + \frac{\nu_i
    q_{i0}+ r_{i,-1} + p_{i1}
    \nu_i(\nu_i-1)}{(z-\zeta_i)} +O(1),\quad z\to \zeta_i,
  \end{align*}
  where $r_{i,-1}$ is the residue of $r(z)$ at $z=\zeta_i$.  Relation
  \eqref{eq:z'x} implies
  \begin{align*}
    (\zeta(x) - \zeta_i)^{-1} &\equiv  \frac{1}{\zeta'(0)}
                                x^{-1}+O(1),\\
    (\zeta(x) - \zeta_i)^{-2} &\equiv \frac{1}{\zeta'(0)^2 }x^{-2} - 
                                \frac{\zeta''(0)}{\zeta'(0)^3
                                }x^{-1}+O(1),\quad \\
                              &\equiv \frac{1}{p_{i0} }x^{-2} -
                                \frac{p_{i1}}{2 p_{i0} 
                                \zeta'(0) }x^{-1}+O(1),
  \end{align*}
  with all relations holding as $x\to 0$.  Hence,
  \[ U(x) \equiv \nu_i(\nu_i+1)x^{-2} - \frac{1}{\zeta'(0) }(\nu_i q_{i0}
  + r_{i,-1}+ \frac12 p_{i,1} \nu_i (3\nu_i-1))x^{-1}+O(1),\quad
  x\to 0. \]
  By Lemma \ref{lem:dg} the coefficient of $x^{-1}$ must
  vanish, which leads directly to \eqref{eq:T-1}.
\end{proof}

\begin{lem}
  \label{lem:pqrdeg}
  If $T$ is exceptional, then
  $\deg p \leq 2, \deg q\leq 1, \deg r\leq 0$.
\end{lem}
\begin{proof}
  Use polynomial division to obtain the following decompositions
  \[ q(z) = q_{\mathrm{p}}(z) + q_{\rms}(z),\qquad r(z)=
  r_{\mathrm{p}}(z) + r_{\rms}(z),\] where $q_{\mathrm{p}},
  r_{\mathrm{p}}\in \cP$  and  $q_{\rms},
  r_{\rms}\in \cQ$ with
  \[  \deg q_{\rms}, \deg r_{\rms} < 0,\qquad  \deg q_{\mathrm{p}}=\deg q,\qquad \deg r_{\mathrm{p}}=\deg r .\]  Next
  consider the decomposition $T= T_{\mathrm{p}}+ \TS$, where
  \[ T_{\mathrm{p}} = p(z) D_{zz} + q_{\mathrm{p}}(z) D_z +
  r_{\mathrm{p}}(z),\quad \TS = q_{\rms}(z) D_z+ r_{\rms}(z).\] By
  construction, 
  \[ \deg \TS[y] < \deg y,\quad y\in \cP.\]
  Hence, if $y_k\in \cP$ is an eigenpolynomial of degree $k$ we must
  have
  \[ \deg T_{\mathrm{p}}[y_k] \leq k.\]
  The desired conclusion follows because this is true for infinitely
  many $k$.
\end{proof}

\begin{lem}
  \label{lem:redform}
  Suppose that $T$ is reduced and exceptional, and let
  \begin{align}
    \label{eq:etaprod}
    \eta(z) &= \prod_{i=1}^N (z-\zeta_i)^{\nu_i},\\
    \label{eq:muprod}
    \mu(z) &= \prod_{i=1}^N (z-\zeta_i)^{\nu_i(\nu_i-1)/2},
  \end{align}
  where $\zeta_i, i=1,\ldots, N$ are the poles of $T$, and
  $\nu_i = \nu_{\zeta_i}$ the corresponding gap cardinalities as per Definition
  \ref{def:orddegseq}.  Then, for some $s\in \cP_1$ and $c\in \Cset$
  we have
  \begin{subequations}
    \label{eq:genqr}
    \begin{align}
      \label{eq:Tqform}
      q &= \frac12 p'+s-
          \frac{2p\eta'}{\eta}\\
      \label{eq:Trform}
      r &=  \frac{p\eta''}{\eta} + \lp \frac{p'}{2} -
          s\rp \frac{\eta'}{\eta}+2p\left(\frac{\mu''}{\mu} -
          \left(\frac{\mu'}{\mu}\right)^2\right) + \frac{p' \mu'}{\mu}+c.
    \end{align}
  \end{subequations}
\end{lem}

Note that a reduced operator $T$ would also be natural if $\mu=1$, i.e. if $\nu_i=1,\quad i=1,\dots,N$. If some $\nu_i>1$ then a gauge transformation is needed to map the reduced $T$ into natural form, as we see below. In practice, exceptional operators with poles $\nu_i>1$ exist in the Laguerre and Jacobi cases, but only for a set of null measure in the parameters.

\begin{proof}
  By Lemma \ref{lem:qpoles} and by \eqref{eq:T-2} of Lemma
  \ref{lem:ordergaps} we have
  \begin{align}
    \label{eq:qform1}
    q(z) &\equiv -\frac{2 p_{i0}\nu_i}{z-\zeta_i} + O(1),\quad
        z\to\zeta_i
  \end{align}
  where $p_{i0} = p(\zeta_i)$.  
  Set $s:= q- p'/2+2p\eta'/\eta$, so that relation \eqref{eq:Tqform}
  holds. By \eqref{eq:etaprod},
  \[ \frac{\eta'(z)}{\eta(z)} \equiv \frac{\nu_i}{z-\zeta_i}
  +O(1),\quad z\to \zeta_i.\]
  Hence, $s(z)$ has vanishing residues at all primary poles
  $z=\zeta_i$.  By Lemma \ref{lem:pqrdeg},
  $\deg p\leq 2,\;\deg q\leq 1$, which implies that $s\in \cP_1$. 

  Let $\tr(z)$ denote the right side of \eqref{eq:Trform}.  Since
  $\deg p\leq 2$, by inspection, $\deg \tr\leq 0$.  By Lemma
  \ref{lem:pqrdeg}, $\deg r\leq 0$.  By Lemma \ref{lem:redpoles},
  $r(z)$ has simple poles at $z=\zeta_i,\; i=1,\ldots, N$.  Hence,
  relation \eqref{eq:Trform} will follow once we show that $r(z)$ and
  $\tr(z)$ have the same residues at all $z=\zeta_i$. Set
  \[ \tau_i = \sum_{j\neq i} \frac{\nu_j}{\zeta_i-\zeta_j},\quad
  i=1,\ldots, N,\]
  so that
  \begin{align*}
    \frac{\eta'(z)}{\eta(z)} 
    &\equiv \frac{\nu_i}{z-\zeta_i} + \tau_i +O( (z-\zeta_i)),\\
    p(z) \frac{\eta'(z)}{\eta(z)} 
    &\equiv (p_{i0} + p_{i1} (z-\zeta_i))\left(\frac{\nu_i}{z-\zeta_i}
      + \tau_i\right)  +O( (z-\zeta_i)),\\
    &\equiv \frac{p_{i0}\nu_i}{z-\zeta_i} + p_{i0}\tau+p_{i1} \nu_i
      +O((z-\zeta_i))\quad     z\to \zeta_i.
  \end{align*}
  From \eqref{eq:Tqform}, which we have already established, it follows that
  \[ q_{i0} = p_{i1}\left(\frac12 -2 \nu_i\right)+ s_{i0}
  -2p_{i0}\tau_i, \qquad s_{i0} = s(\zeta_i).\]
  and  by \eqref{eq:T-1} of Lemma \ref{lem:T-1} we have
 \begin{align}
    \label{eq:rform1}
   r(z)  &\equiv  \frac{\frac12 p_{i1} \nu_i(1-3\nu_i) -
           \left(p_{i1}\left(\frac12 -2 \nu_i\right)+ s_{i0}
           -2p_{i0}\tau_i \right)\nu_i}{z-\zeta_i}+O(1)
           ,\quad
           z\to \zeta_i
 \end{align}
  Hence by \eqref{eq:muprod}   and a direct calculation we obtain
   \begin{align*}
     \frac{\mu'(z)}{\mu(z)} 
     &\equiv \frac{\frac12 \nu_i(\nu_i-1)}{z-\zeta_i}+O(1)\\
     \frac{\mu''(z)}{\mu(z)} - \left(\frac{\mu'(z)}{\mu(z)}\right)^2
     &\equiv -\frac{\frac12 \nu_i(\nu_i-1)}{(z-\zeta_i)^2}+O(1),\\
     \frac{\eta'(z)}{\eta(z)} +\frac{2\mu'(z)}{\mu(z)}
     &\equiv \frac{ \nu_i^2}{z-\zeta_i}+O(1),\\
     \frac{\eta''(z)}{\eta(z)} 
     &\equiv \frac{\nu_i(\nu_i-1)}{(z-\zeta_i)^2}
       + \frac{2\tau_i \nu_i}{z-\zeta_i} +O(1),\\
     \tr(z)&=p\lp\frac{\eta''}{\eta} +
             2\lp\frac{\mu''}{\mu}-\left(\frac{\mu'}{\mu}\right)^2\rp\rp+ 
             p'\lp \frac{\eta'}{2\eta}+\frac{\mu'}{\mu}\rp - 
             \frac{s\eta'}{\eta}   \\
     &\equiv \frac{2p_{i0}\tau_i \nu_i+\frac12 p_{i1}\nu_i^2 -s_{i0} \nu_i}{z-\zeta_i} +O(1),\quad z\to \zeta_i,
  \end{align*}
  which agrees with \eqref{eq:rform1}.
\end{proof}

We now show that the operator form shown in \eqref{eq:genqr} is
gauge-invariant.
\begin{lem}
  \label{lem:genqrxform}
  Let $T=p(z)D_{zz} +q(z) D_z+ r(z)$ where  $p\in \cP_2$ and
 \begin{subequations}
   \label{eq:genqrform}
    \begin{align}
      q &= \frac{p'}{2}+ s-
          \frac{2p\eta'}{\eta}\\ \label{eq:genrform}
      r &=  \frac{p\eta''}{\eta} + \lp \frac{p'}{2} -
           s\rp \frac{\eta'}{\eta}+2p\left(\frac{\mu''}{\mu} -
          \left(\frac{\mu'}{\mu}\right)^2\right) + \frac{p' \mu'}{\mu}.
    \end{align}
  \end{subequations}  
  for some $s\in \cP_1$ and $\eta,\mu\in\cQ$.  Let $\sigma \in \cQ$,
  and let $\tT = \sigma T \sigma^{-1}$ be the indicated,
  gauge-equivalent operator.  Then the coefficients $\tq(z),\tr(z)$ of
  $\tT$ have the form shown in \eqref{eq:genqrform}, with
  \begin{equation}\label{eq:pqrtransf}
    \teta = \sigma\eta, \quad \tmu = \sigma^{-1} \mu,
  \end{equation}
  in place of $\eta, \mu$.
\end{lem}
\begin{proof}
  Set
  \[ H = \frac{\eta'}{\eta},\quad
  M= \frac{\mu'}{\mu},\quad
  S=\frac{\sigma'}{\sigma},\] 
  \[ \tH=\frac{\teta'}{\teta}=H+S,\quad \tM=\frac{\tmu'}{\tmu}
  = M-S.\] 
  Applying \eqref{eq:pqrWgaugelaw}, we have
  \begin{align*} 
    \tq &= q-2pS \\
        &= \frac{p'}{2}+s-2pH-2pS \\
        &= \frac{p'}{2}+s-2\tH,\\
    \tr &= r - q S + p (-S'+ S^2), \\
        &=p(H'+H^2) + \lp \frac{p'}{2} - s\rp H+2p M' + p' M-\lp
          \frac{p'}{2} + s - 2 p H\rp S +p(- S' + S^2), \\
    &=p(H'+S'+(H+S)^2+2M'-2S') + p'\lp \frac{H}{2}+\frac{S}{2}+M-S\rp
      -s (H+S), \\
    &=p(\tH'+\tH^2) + \lp\frac{p'}{2}-s\rp \tH +2p\tM'+p'\tM,
  \end{align*}
  which is the form shown in \eqref{eq:genqrform} but with $\eta, \mu$
  replaced by $\teta,\tmu$.
\end{proof}

\begin{proof}[Proof of Theorem \ref{thm:natural}.]
  Let $\tT= p(z) D_{zz} + \tq(z) D_z + \tr(z)$ be an exceptional
  operator with maximal invariant polynomial subspace $\tilde \cU$.
  By Proposition \ref{prop:greduced}, let $\sigma\in\cP$ be a GCD of
  all polynomials in $\tilde \cU$ so that $T = \sigma^{-1} \tT \sigma$
  is reduced.  Lemma \ref{lem:redform} gives the form of $T$. By Lemma
  \ref{lem:genqrxform}, $\tT$ has the same form.  Let
  $\zeta_1,\ldots, \zeta_N$ be the poles of $T$ and
  $\nu_1,\ldots, \nu_N$ the gap cardinalities as per Definition
  \ref{def:orddegseq}.  Write
  \[ \sigma(z) = \prod_{i=1}^{N} (z-\zeta_i)^{\alpha_i}
  \prod_{i=1}^{M} (z-\xi_i)^{\beta_i}\]
  where $\xi_1,\ldots, \xi_M$ are the zeros of $\sigma(z)$ distinct
  from the $\zeta_i$, and $\alpha_i\geq 0, \beta_i\geq 0$ the
  corresponding multiplicities.  Let $\cU$ be the maximal invariant
  polynomial subspace of $T$. 

  We claim that $\tcU = \sigma \cU$.  The inclusion
  $\sigma \cU \subset \tcU$ is obvious.  We now prove that
  $\tcU \subset \sigma \cU$. As was shown in the proof of Proposition
  \ref{prop:greduced}, every element of $\tcU$ is divisible by
  $\sigma$.  Let $\ty\in \tcU$ be given and set
  $y=\sigma^{-1}\ty$. Observe that
  \[ T^k[y] = (\sigma^{-1} \tT^k\sigma)[y] = \sigma^{-1}
  \tT^k[\ty],\quad k\in \Nset.\]
  By definition, $\tT^k[\ty]\in \tcU$ for all $k\in \Nset$.  Hence,
  $T^k[y]\in \cP$ for all $k$.  Therefore, $y\in \cU$ by Proposition
  \ref{prop:Udef}.

  Having established the claim, we infer that the poles and the gap
  cardinalities of $\tT$ are
  \[ \tzeta_i =
  \begin{cases}
    \zeta_i, & i = 1,\ldots, N\\
    \xi_{i-N}, & i = {N+1},\ldots, N+M
  \end{cases},\qquad
\tnu_i =
  \begin{cases}
    \nu_i + \alpha_i,& i=1,\ldots, N\\
    \beta_{i-N}, & i={N+1},\ldots, N+M
  \end{cases}\] By Lemmas \ref{lem:redform} and \ref{lem:genqrxform},
  \begin{align*}
    \tq(z) &\equiv -2 \sum_{i=1}^{N}
             \frac{p(\zeta_i)\nu_i}{z-\zeta_i}-2 \sum_{i=1}^{N}
             \frac{p(\zeta_i)\alpha_i}{z-\zeta_i} - 2 \sum_{i=1}^M
             \frac{p(\xi_i)\beta_i}{z-\xi_i} \mod \cP_1\\
           &\equiv -2 \sum_{i=1}^{N}
             \frac{p(\zeta_i)(\nu_i+\alpha_i)}{z-\zeta_i} - 2 \sum_{i=1}^M
             \frac{p(\xi_i)\beta_i}{z-\xi_i} \mod \cP_1\\
           &\equiv -2 \sum_{i=1}^{N+M}
             \frac{p(\tzeta_i) \tnu_i}{z-\tzeta_i} \mod \cP_1,
  \end{align*}
  which proves the first assertion of the Theorem.

  Next, set $\hT= \mu T \mu^{-1}$, with $\mu$ as per
  \eqref{eq:muprod}.  Let $\hq,\hr$ be the corresponding first- and
  zero-order coefficients.  By Lemma \ref{lem:genqrxform}, $\hr$ has
  the form shown in \eqref{eq:genrform}, but with $\heta = \mu \eta$
  and $\hmu = 1$ in place of $\eta,\mu$.  Hence,
  \[ \hr = \frac{p\heta''}{\heta} + \lp \frac{p'}{2} - s\rp
  \frac{\heta'}{\heta},\]
  which proves the second assertion of the Theorem.

%
\end{proof}

Before moving on to the next section, we make a remark and state two
corollaries of Theorem \ref{thm:natural} that generalize results for
exceptional Hermite polynomials previously established in
\cite{Gomez-Ullate2015}.  These results are not used elsewhere in the
paper, but they may have some significance for future research,  in particular for the derivation of recurrence relations for exceptional polynomials.

\begin{remark}
  Since the roots of the indicial equation at a primary pole and at an
  ordinary point are non-negative, the general solution of
  $T[y]=\lambda y$ is not only meromorphic but holomorphic at such
  points.  The only points at which the general solution of
  $T[y]=\lambda y$ might not be meromorphic are the secondary poles of
  $T$, i.e. the roots of $p(z)$. In the case $p(z)=1$ which corresponds to exceptional Hermite
  operators, the general solution is thus an entire function, as
  proved in \cite{Gomez-Ullate2014}.
\end{remark}

\begin{cor}
  \label{cor:naturalU}
  Let $T\in \Diff_2(\cQ)$ be a natural exceptional operator with poles $\zeta_1,\dots,\zeta_N$ and corresponding gap multiplicities $\nu_1,\dots,\nu_N$. Let $\cU\subset \cP$ be the maximal polynomial invariant subspace of $T$, and  $\eta\in \cP$ be given by $\eta(z) = \prod_{i=1}^N (z-\zeta_i)^{\nu_i}$.
Then $y\in \cU$ if and  only if
  \begin{equation}
    \label{eq:Unatural}
2p\eta' y' -\left(p\eta'' + \frac{1}{2}
      p'\eta'-s\eta'\right) y 
  \end{equation}
  is divisible by $\eta$.
\end{cor}
\begin{proof}
  Let $\cU'\subset \cP$ be the polynomial subspace consisting of those
  $y\in \cP$ such that \eqref{eq:Unatural} is divisible by $\eta$.  If
  $y\in \cU$, then $T[y]\in \cP$ by Proposition \ref{prop:Udef}.
  Decompose the operator in \eqref{eq:natqr} as $T=T_0 + \TS$ where
  \begin{align*}
    T_0 &=  p D_{zz} + \left(\frac{p'}{2} + s\right)D_z\\
    \TS &= - \frac{2p\eta'}{\eta} D_z + \frac{p\eta''}{\eta} +
    \left(\frac{p'}{2} -s\right) \frac{\eta'}{\eta}.
  \end{align*}
  Since $T_0$ has polynomial coefficients, $\TS[y] \in \cP$.
  Hence, $y\in \cU'$, and therefore $\cU\subset \cU'$.

  To obtain equality, we use a codimension argument.  For $i=1,
  \ldots, N,\; j=0,\ldots, \nu_i-1$, define the differential
  functionals $\alpha^{(j)}_i : \cP \to \Cset$ by
  \[ y\mapsto D_z^j\left( 2p(z)\eta'(z) y'(z) -\left(p(z)\eta''(z) +
      \frac{1}{2} p'(z)\eta'(z)-s(z)\eta'(z)\right)y(z) \right)
  \Big|_{z=\zeta_i}.\] 
  Observe that $y\in \cP$ is divisible by $\eta$ if and only if
  \[ y^{(j)}(\zeta_i) = 0 \] for the range of $i,j$ given above.
  Hence, $\cU'$ is the joint kernel of the $\alpha^{(j)}_i$.  By
  Proposition \ref{prop:dfuncindep}, these functionals are linearly
  independent, and therefore $\cU'$ has codimension $\sum_{i=1}^N \nu_i$
  in $\cP$.  By Theorem \ref{thm:codimsum}, this is also the
  codimension of $\cU$ in $\cP$, so we must have $\cU = \cU'$.
\end{proof}

\begin{cor}
  Let $T\in \Diff_2(\cQ)$ be a natural exceptional operator, $\cU$ its maximal invariant polynomial subspace and $\eta$ be the polynomial defined in \eqref{eq:etadef}.  Suppose
  that $f\in \cP$ is such that $f'$ is divisible by $\eta$. Then,
  multiplication by $f$ preserves $\cU$; i.e., $fy\in \cU$ for every
  $y\in \cU$.
\end{cor}
\begin{proof}
  Suppose that $f'$ is divisible by $\eta$.
  Replacing $y$ with $fy$ in \eqref{eq:Unatural} yields 
  \[ 2p\eta' (fy)' -\left(p\eta'' + \frac{1}{2} p'\eta'-s\eta'\right)
  fy = f\bigg[2p\eta' y' -\left(p\eta'' + \frac{1}{2} p'\eta'-s\eta'\right)
  y \bigg]+ 2p\eta' f' y.\] By Corollary \ref{cor:naturalU}, if $y\in \cU$,
  then the above is divisible by $\eta$.
\end{proof}

The above Corollary allows to build recurrence relations for exceptional polynomials, where the traditional multiplication by $x$ is substituted by multiplication by the polynomial $f$ satisfying the above condition, \cite{Miki2015,Gomez-Ullate2015,Odake2016,Duran2015a}. The smallest order recurrence relations are obtained by taking $f=\int \eta$, the anti-derivative of $\eta$. Since $\deg\eta=\nu$, these will be recurrence relations of order $2\nu+3$.

\section{Proof of the Conjecture}
\label{sec:L}

In this section we prove the previously conjectured result that every
exceptional operator is Darboux connected to a classical operator.  We
begin with some preliminaries.
\begin{definition}
  For $L \in \Diff_\rho(\cQ)$ we define the degree of $L$ to be
  \begin{equation}
    \label{eq:degLy}
    \deg L= \max \{  \deg a_j -j \colon j = 0,1,\ldots, \rho \},
  \end{equation}
  where the $a_j\in \cQ$ is the $j^{\text{th}}$ order coefficient as
  per \eqref{eq:Lcoeffs}.
\end{definition}

The degree of an operator has an alternative, but equivalent
characterization. Let $L\in \Diff_\rho(\cQ)$ and $k = \deg L$, as
defined above.  Express the coefficients of $L$ as
\[ a_j(z)\equiv c_j z^{j+k} \mod \cQ_{j+k-1},\quad c_j\in \Cset.\]
 and  define the polynomial
\[ \sigma(n) = \sum_{j=0}^\rho c_j n(n-1)\cdots (n-j+1). \] 
\begin{prop}
  \label{prop:degL}
  The degree of an operator $L\in \Diff(\cQ)$ is the smallest integer
  $k$ such that $\deg L[y] \leq k + n$ for all $y\in \cQ_n$.
\end{prop}

\begin{prop}
  \label{prop:Lsymbol}
  We have
  \[ \deg L[y] \leq \deg L + \deg y,\quad y\in \cQ .\]
  The inequality is strict if and only if $\deg y$ is a zero of
  $\sigma$.
\end{prop}
\begin{proof}
  It suffices to show that
  \[ L[z^n] \equiv \sigma(n) z^{n+k}\mod \cQ_{n+k-1},\quad n\in
  \Nset.\] Write $L = L_0 + L_1$, where
  \[ L_0 = \sum_{j=0}^\rho c_j z^{k+j} D_z^j,\]
  is a homogeneous degree $k$ operator. Hence, $\deg L_1 < k$ by
  construction, and 
  \[ \deg L_1[z^n]  \leq n+k-1,\quad n\in \Nset .\]
  The desired conclusion follows once we observe that
  \[ L_0[z^n] = \sigma(n) z^{n+k},\quad n\in \Nset.\]
\end{proof}

\begin{definition}
  We  say that $T\in \Diff_2(\cP)$ is a Bochner operator (or classical operator) if $\deg T=0$.
\end{definition}

Before stating the main result of this section, we note the following.
\begin{prop}
  Every Bochner operator is exceptional.
\end{prop}
\begin{proof}
  Let $T$ be a Bochner operator. By  Proposition \ref{prop:Lsymbol}
  \ref{prop:degL},
  \[ T[z^k] \equiv \sigma(k) z^k\mod \cP_{k-1}\]
  where $\sigma(k)$ is a non-zero polynomial of degree $\leq 2$.
  Hence, $T-\sigma(k)$ maps $\cP_k$ into $\cP_{k-1}$ for every
  $k\in \Nset$.  By the rank-nullity theorem, this linear map has a
  non-trivial kernel, which means that, for every $k\in \Nset$, there
  exists a $y_k \in \cP_k$ such that
  \[ T[y_k] = \sigma(k) y_k.\] However $\deg y_k$ may be strictly less
  than $k$, which means that $y_k$ may coincide with an
  eigenpolynomial of lower degree.  However, this can happen only if
  $\sigma(k) = \sigma(k')$ for some $k'\neq k$; i.e. if the eigenvalue
  is not simple.  Since $\sigma(k)$ is at most a quadratic function, and
  $k$ is a positive integer, this can happen at most finitely many
  times.  Therefore, a co-finite number of eigenvalues $\sigma(k)$ are
  simple, which means that there are eigenpolynomials for a co-finite
  number of degrees $k$. Therefore, $T$ is an exceptional operator
  according to Definition \ref{def:exT}.
\end{proof}

\begin{remark}
Note that Bochner operators need not have polynomial eigenfunctions for every degree $k\in \mathbb N$. See for example Remark \ref{rem:ss} and a counter-example in Example \ref{ex:1}.
\end{remark}

The main result of this section is the following theorem.

\begin{thm}\label{thm:inter}
  Let  $T\in \Diff_2(\cQ)$  be an exceptional operator with primary poles  $\zeta_1,\dots,\zeta_N$ and
    corresponding gap cardinalities $\nu_1,\ldots, \nu_N$. Then, $T$ is Darboux connected to a Bochner operator $\TB \in \Diff_2(\cP)$.  Moreover, if
  $p\in \cP_2$ is the second-order coefficient of $T$, and  $W,\WB$ the weights associated by
  \eqref{eq:Wdef} to $T,T_B$, we have the relation
  \begin{equation}
    \label{eq:WW0}
    W(z) = \WB(z) \frac{\chi(z)}{\eta(z)^2},\quad \chi\in\cQ,\; \eta\in \cP,
  \end{equation}
  where 
  \begin{equation}
    \label{eq:pchi'}
\eta(z)= \prod_{i=1}^N (z-\zeta_i)^{\nu_i},\qquad  \frac{\chi'(z)}{\chi(z)}  = \frac{k}{p(z)},\quad k\in \Cset.
  \end{equation}
\end{thm}



\noindent
The proof of Theorem \ref{thm:inter} requires a number of preliminary
results.  Let $T\in \Diff_2(\cQ)$ be an exceptional operator and
consider the vector space
\[ \cL:=\{ L\in\Diff(\cP)\,:\,T^kL\in \Diff(\cP) \text{ for all } k\in
\Nset \}. \] The following is an equivalent characterization of $\cL$.
\begin{lem}\label{lem:LPU}
  For $L\in \Diff(\cP)$, we have $L\in \cL$ if and only if $L[\cP]
  \subset \cU$.
\end{lem}
\begin{proof}
  One direction is trivial; we prove the converse.  Suppose that
  $L\in\cL$ so that we have $T^k\big[L[y]\big]\in \cP$ for all $y\in
  \cP$ and all $k\geq 1$.  By Definition \ref{def:U}, this implies that
  $L[y] \in \cU$, as was to be shown.
\end{proof}

Next, define the subspace
\[ \cL^{(\rho)}:=\{ L\in \cL\,\colon\, \ord L\leq \rho,\,\deg L\leq
0\}\] where it is clear that $\cL^{(\rho_1)}\subset \cL^{(\rho_2)}$
for $\rho_1<\rho_2$. We will first show that at least one
$\cL^{(\rho)}$ is non-trivial.
\begin{lem}
  Let $\zeta_1,\dots,\zeta_N$ be the primary poles of $T$, and
  $\nu_1,\ldots, \nu_N$ the corresponding gap cardinalities.  Then,
  $\dim \cL^{(n)}>0$ where
  \[ n = \sum_{i=1}^N 2\nu_i.\]
\end{lem}
\begin{proof}
  Set
  \begin{equation}
    \label{eq:etandef}
    f(z) = \prod_{i=1}^N (z-\zeta_i)^{2\nu_i}.
  \end{equation}
  By construction, for every $y\in\cP$
  \[ \alpha_{ki}[f y] = 0,\qquad i=1,\dots,N,\; k\notin
  I_{\zeta_i},\]
  where $\{\alpha_{ki}\}_{k\notin I_{\zeta_i}}$ is the basis of
  $\Ann_{\zeta_i}\cU$ defined in \eqref{eq:alpha}. Hence, by the proof
  of Proposition \ref{prop:fc}, $f y\in \cU$ for all $y\in \cP$,
  and Lemma \ref{lem:LPU} implies that the differential operator
  \[ L = f(z) D_z^{n},\] belongs to $\cL$.  By Proposition
  \ref{thm:codimsum} its degree is zero, so $L\in \cL^{(n)}$ as was
  to be proved.
\end{proof}

Now, let $\rmin$ be the minimum positive integer such that $\dim
\cL^{(\rho)}>0$, i.e. $\dim \cL^{(\rmin)}>0$ but $\dim \cL^{(\rho)}=0$
for all $\rho<\rmin$. 
\begin{lem}
  \label{lem:degL0}
  For all non-zero $L \in \cL^{(\rmin)}$ we have $\ord L = \rmin$ and
  $\deg L = 0$ exactly.
\end{lem}
\begin{proof}
  The order equality holds by the minimality assumption on $\rmin$.
  Similarly, suppose that there exists a non-zero $L\in \cL^{(\rmin)}$
  such that $\deg L = -d<0$. Since $L$ has polynomial coefficients,
  such an operator would necessarily be of the form $L=\tL D^d$, where
  $\tL \in\Diff(\cP)$.  This would imply that $\tilde L\in
  \cL^{(\rmin-d)}$, which would again contradict the minimality
  assumption for $\rmin$.
\end{proof}
\begin{lem}
  \label{lem:dimLrmin}
  $\dim \cL^{(\rmin)}\leq \rmin+1$.
\end{lem}
\begin{proof}
  Observe that $\rmin+1$ is the dimension of the space of degree
  homogeneous differential operators of order $\rmin$.  Hence if
  $\dim \cL^{(\rmin)}$ were to exceed this bound, we would be able
  to construct an operator $L\in\cL^{(\rmin)}$ having strictly negative
  degree, which is impossible by Lemma \ref{lem:degL0}. 
\end{proof}
\begin{lem}
  \label{lem:TT0T1}
  Let $T$ be an exceptional operator. Then, there exist a
  decomposition
  \[ T = T_0 +  \TS,\] where $T_0\in \Diff_2(\cP)$ is a Bochner
  operator, and $\TS\in \Diff_1(\cQ)$ has negative degree.
\end{lem}
\begin{proof}
  Let $p,q,r$ be the coefficients of $T$, as per \eqref{eq:Tgeneral}.
  By Lemmas \ref{lem:qpoles} and \ref{lem:pqrdeg}, we can write
  \[q=q_1 + q_{\rms},\quad r=r_0 + r_{\rms},\]
  with $q_1\in \cP_1$, $r_0\in \Cset$ , $q_{\rms},r_{\rms}\in \cQ$,
  with
  \[  \deg q_{\rms}, \deg  r_{\rms} <0.\]  Taking
  \[ T_0 = p D_{zz} + q_1 D_z + r_0,\qquad \TS = q_s D_z + r_{\rms} \]
  gives the desired decomposition.
\end{proof}

\begin{lem}
  \label{lem:TLLT}
  Let $T$ be an exceptional operator and $T_0, \TS$ its
  decomposition into Bochner and singular part according to Lemma
  \ref{lem:TT0T1} . If $L\in \cL^{(\rmin)}$ is non-zero, then
  \begin{equation}\label{eq:TLLT}
    \deg \left(T L - L T_0\right)<0.
  \end{equation}
\end{lem}
\begin{proof}
  By Lemma \ref{lem:degL0}, $\deg L= 0$.  Hence, for $y\in \cP_n$ we
  have
  \begin{align}
    T[y] &\equiv \sigma_1(n) \,y ,\mod \cQ_{n-1}\\
    T_0[y] &\equiv \sigma_1(n) \,y ,\mod \cQ_{n-1}\\
    L[y] &\equiv \sigma_2(n) \,y \mod \cQ_{n-1}.
  \end{align}
  where, $\sigma_1(n), \sigma_2(n)$ are polynomials defined by
  Proposition \ref{prop:Lsymbol}. Hence,
  \begin{align}
    (TL)[y] &\equiv  T[ \sigma_2(n) \,y ] \equiv \sigma_1(n)
              \sigma_2(n) y ,\mod \cQ_{n-1}\\ 
    \left(LT_0\right)[y] 
            &\equiv L[ \sigma_1(n) \,y ] \equiv \sigma_2(n)
              \sigma_1(n)y \mod \cQ_{n-1},
  \end{align}
  which establishes \eqref{eq:TLLT}.
\end{proof}

\begin{lem}\label{lem:cA}
    Let $T$ be an exceptional operator and $T_0, \TS$ its decomposition according to Lemma \ref{lem:TT0T1}. Then, there exists a
  linear transformation $\cA: \cL^{(\rmin)} \to \cL^{(\rmin)}$ such
  that
  \[ \cA(L)D = TL- L T_0,\quad L\in\cL^{(\rmin)}.\]
\end{lem}
\begin{proof}
  Since $T$ is second-order,
  \[ \ord (TL-LT) \leq \ord L+1,\quad L\in \Diff(\cP).\]
  By construction, $T-T_0$ is a first-order operator, and hence
  \[ \ord (TL-L T_0) \leq \ord L+1,\quad L \in \Diff(\cP)\] also. By
  Lemma \ref{lem:TLLT}, if $L\in \cL^{(\rmin)}$, then
  \[ TL-L T_0 = \tL D \] for some unique operator $\tL\in
  \Diff(\cP)$. By construction, $(\tL D)[\cP]\subset\cU$ which means
  that that $\tL[\cP]\subset\cU$ as well. Hence, Lemma
  \ref{lem:LPU} implies that $\tL\in\cL$. If $L\in \cL^{(\rmin)}$
  then by the above results we see that $\ord \tL\leq\ord L=\rmin$ and
  $\deg\tL\leq 0$. This implies that $\tilde L\in\cL^{(\rmin)}$.  Our
  claim is established once we set $\cA(L):= \tL$.
\end{proof}

\begin{proof}[Proof of Theorem \ref{thm:inter}]
  Let $T,T_0$ be as in the preceding Lemma.
  By Lemma \ref{lem:dimLrmin}, $\cL^{(\rmin)}$ is finite
  dimensional. Hence, there exists an eigenvector $L\in \cL^{(\rmin)}$
  with eigenvalue $\gamma$ of the linear transformation $\cA$ defined
  in Lemma \ref{lem:cA}. This means that $L\in\Diff(\cP)$ and $\cA(L)
  = \gamma L$ so that
  \[ TL=L(T_0 +\gamma D).\]
  Therefore $\TB=T_0+\gamma D$ is the desired Bochner operator.

 Since $T$ is an exceptional operator, from Theorem \ref{thm:natural} it follows that its first order coefficient $q(z)$  is given by \eqref{eq:natqform}, with $\eta$ determined by \eqref{eq:etadef}.
  The weight $W(z)$ is then determined by \eqref{eq:PWRdef} to be,
  \[ W(z) = \exp\left(\int^z \frac{a x + b}{p(x)} dx\right)
  \eta(z)^{-2}, \quad a,b\in \Cset. \]
  By definition of  $T_0$ in Lemma \ref{lem:TT0T1} we see that the weight $\WB(z)$  associated to $\TB$ must have the form
  \[ \WB(z) = \exp\left(\int^z \frac{a x + c}{p(x)} dx\right),\quad
  c\in \Cset.\]
  By \eqref{eq:hWW} of Proposition \ref{prop:hWW}, $W(z)/\WB(z)$ is a
  rational function, which implies that
  \[\chi(z)= \eta(z)^2 \frac{W(z)}{\WB(z)} = \exp\left(\int^z
    \frac{b-c}{p(x)}dx\right)\]
  is a rational function.  Therefore, by inspection, \eqref{eq:pchi'}
  holds with $k=b-c$.
\end{proof}

\begin{remark}
  Observe that in Lemma \ref{lem:TT0T1} the decomposition $T=T_0+\TS$
  is not unique.  Indeed, for every $\gamma_0 \in \Cset$ the operators
  \[ T'_0 = T_0+\gamma_0 D,\quad \TS'= \TS-\gamma_0 D \]
  give another valid decomposition of $T=T_0+\TS$ into Bochner and
  degree-lowering summands.  The eigenvalue $\gamma$ utilized in the
  above proof then undergoes a corresponding shift to compensate for
  this: $\gamma' = \gamma-\gamma_0$.
\end{remark}

\section{Exceptional Orthogonal Polynomial Systems}\label{sec:OPS}

In all of the previous sections the differential operator $T$ was
treated at a purely formal level, the emphasis being on the algebraic
conditions leading to the existence of an infinite number of
polynomial eigenfunctions. In this section, analytic conditions will
be further imposed, in order to select those operators that have a
self-adjoint action on a suitably defined Hilbert space.

\begin{definition}
  \label{def:ss} Let $T$ be an exceptional operator.  We say that
  $T$ is \emph{polynomially semi-simple} if the action of $T$ on every
  finite-dimensional, invariant polynomial subspace is diagonalizable.
  We will say that $T$ is \emph{polynomially regular} if there exists
  a positive-definite inner product on $\cP$ relative to which the
  action of $T$ is symmetric.
\end{definition}

\begin{remark}
  \label{rem:ss}
  By \eqref{eq:Udef}, $\cU$ contains all eigenpolynomials of $T$,
  which means that $\nu\leq m<\infty$, where $m$ is the number of
  exceptional degrees as per Definition \ref{def:exT}. If $T$ is also
  polynomially semi-simple, then $\cU$ may be characterized as the
  span of the eigenpolynomials of $T$, in which case $\nu=m$.
  However, in general $\cU$ may contain polynomials that are not in
  the span of the eigenvectors of $T$, in which case $\nu<m$ strictly.

  The polynomial semi-simplicity condition has not been considered
  previously in the literature. Rather in the context of orthogonal
  polynomial systems, the usual assumption is that $T$ is related to a
  Sturm-Liouville operator with polynomial eigenfunctions, which under
  suitable assumptions, detailed below, implies that $T$ is
  polynomially regular.  By the finite-dimensional Spectral Theorem,
  if $T$ is polynomially regular, as per Definition \ref{def:ss}, then
  the $T$-action on invariant, finite-dimensional, polynomial
  subspaces is diagonalizable.  In other words, regularity implies
  semi-simplicity.
  \end{remark}
  To illustrate the above remark, consider the following example.
  
\begin{example}\label{ex:1}
  
  The operator
  \[ T[y] = (1-z^2) y'' + 2(z-2)y' \] is the $\alpha=0,\beta=-4$
  instance of the classical Jacobi operator.  This instance is
  degenerate, because the leading coefficient of the classical Jacobi
  polynomials is
  \[ P_n^{\alpha,\beta}(z) = \binom{\alpha+\beta+2n}{n} 2^{-n} z^n +
  O\left(z^{n-1}\right),\quad z\to \infty. \]
  Indeed, with the above choice of the $\alpha,\beta$ parameters, the
  third-degree Jacobi polynomial $P_3^{\alpha,\beta}$ degenerates to a
  constant.  The constant $y=1$ is an eigenfunction, but observe that
  \[ T[z^3+6z^2+21z] = -72.\] Hence, the vector space spanned by
  $z^3+6z^2+21z$ and $1$ is $T$-invariant, but the action is not
  diagonalizable.  However, the Jacobi polynomials of all other
  degrees are eigenfunctions, so $T$ does fit the definition of an
  exceptional operator.  Regularity for Jacobi polynomials requires
  that $\alpha,\beta >-1$.  Since our example violates this
  assumption, there is no well-defined inner product.  This lack of an
  inner-product permits an operator with an action that is not
  semi-simple.   Thus in this example, $\cU=\cP$ but there is no eigenvector of
  degree $3$, so $m=1$ but $\nu=0$.
\end{example}

The above remarks motivate the following.
\begin{definition}\label{def:SLOPS}
  We say that a co-finite, real-valued polynomial sequence $y_k\in
  \Rset\cP^*_k,\; k\notin \{ k_1,\ldots, k_m\}$ forms a
  Sturm-Liouville orthogonal polynomial system (SL-OPS) provided
  \begin{itemize}
  \item[(i)] the $y_k$ are the eigenpolynomials of an operator $T\in
    \Diff_2(\Rset\cQ)$ ,
  \item[(ii)] there is an open interval $I\subset \Rset$ such that
  \begin{itemize}
  \item[(ii-a)] the associated weight function $W(z)$, as defined by
    \eqref{eq:Wdef}, is positive, single valued, and integrable on
    $I$;
  \item[(ii-b)] all moments are finite, i.e.
    \[\int_I z^j W(z)dz <\infty, \qquad j\in\Nset;\]
  \item[(ii-c)]  $y(z)p(z)W(z)\to 0$ at the endpoints of $I$ for every
    polynomial $y\in \cP$ .
  \end{itemize}
\item[(iii)] the vector space
  $\lspan\{ y_k \colon k\in \Nset\setminus\{ k_1,\ldots, k_m\} \}$ is
  dense in the weighted Hilbert space $\rL^2(W(z)dz, I)$.
\end{itemize}
\end{definition}

Assumption (i) means that $T$ is an exceptional operator.  By
Proposition \ref{prop:gauge-equiv} and Theorem \ref{thm:natural}, no
generality is lost if we assume that $T$ is in the natural gauge;
i.e., that T has the form \eqref{eq:bilinear}, where $\eta$ is given
by \eqref{eq:etadef}.

Proposition \ref{prop:Tsym} and
(ii-c) ensures that $T$ is polynomially regular and that $y_k$ are \emph{orthogonal} 
\[\int_I W(z)y_i(z) y_j(z)dz =c_i \delta _{ij}, \qquad i,j\notin
\{ k_1,\ldots, k_m\},\qquad c_i>0.\] 
As it was already mentioned in Remark \ref{rem:ss}, regularity implies semi-simplicity, which means that $\cU$, the maximal invariant polynomial
subspace, coincides with the span of the eigenpolynomials $y_k,\;
k\notin \{ k_1,\ldots, k_m\}$, and $\nu=m$. Therefore, by assumption (iii),
operator $T$ is essentially self-adjoint on $\cU$.

It has already been noted in all examples of exceptional orthogonal
polynomials published in the literature, that the orthogonality weight
for the exceptional OPS is a classical weight multiplied by a rational
function.  This can now be considered as a result.
\begin{prop}\label{prop:Xweight}
  The orthogonality weight $W(z)$ of a SL-OPS has the form
  \begin{equation}\label{eq:W-WB}
    W(z)=\frac{\WB(z)}{\eta(z)^2}
  \end{equation}
  where 
  \[ \WB(z) = \exp\left(\int^z \frac{s(x)}{p(x)} dx\right), \quad p\in
  \Rset\cP_2,\; s\in \Rset\cP_1 \]
  is the weight of a classical OPS, and where $\eta\in \Rset\cP_m^*$.
\end{prop}

\begin{proof}
  Expression \eqref{eq:W-WB} follows by \eqref{eq:natqform} and
  \eqref{eq:PWRdef}. By the SLOPS assumptions, both $W_B(z)$ and
  $\eta(z)$ must be real valued.  Since $T$ is polynomially regular,
  we have $\nu=m$.  Therefore, $\deg \eta =m$ by \eqref{eq:etadef} and
  Theorem \ref{thm:codimsum}.
\end{proof}

\begin{remark}
  The polynomial $s$ above encodes the weight parameters for the
  Laguerre and Jacobi families.  In the case of the Hermite family all
  parameters can be normalized away by means of a scaling and a
  translation.  In the case of Laguerre families one of the parameters
  can be normalized by means of a scaling.
\end{remark}

\begin{remark}
  If an SL-OPS has polynomial eigenfunctions for all degrees,
  i.e. $m=0$ in Definition \ref{def:SLOPS}, then it defines a
  classical orthogonal polynomial system, which up to an affine
  transformation must be Hermite, Laguerre or Jacobi
  \cite{Bochner1929,Lesky1962}.
\end{remark}


Since every SL-OPS has an associated exceptional operator $T$, the
notion of Darboux connectedness for operators can be naturally
extended to SL-OPS.
\begin{definition}
  We say that two SL-OPS are Darboux connected if their associated
  exceptional operators, modulo a multiplicative constant and a
  spectral shift, are Darboux connected as per Definition
  \ref{def:TDarbtrans}.
\end{definition}



The weights associated with a SL-OPS fall into the same three broad
categories as do classical orthogonal polynomials.
\begin{definition}
  We say that a SL-OPS is of, respectively, Hermite, Laguerre, and
  Jacobi type if the corresponding interval $I=(a,b)$ and weight
  $W(z),\; z\in I$ have the form
  \begin{subequations}
    \label{eq:WHLJ}
    \begin{align}
      I&= (-\infty,\infty),& \WH(z) &= \frac{e^{-z^2}}{\eta(z)^2},  \\ 
      I&= (0,\infty)& 
      \WL(z) &= \frac{z^\alpha e^{-z}}{\eta(z)^2},\quad \alpha>-1,\\ 
      I&= (-1,1)& 
      \WJ(z) &= \frac{(1-z)^\alpha (1+z)^\beta}{\eta(z)^2},\quad
      \alpha,\beta>-1,
    \end{align}
  \end{subequations}
  where $\eta\in \Rset\cP$ is a real-valued polynomial which is
  non-vanishing on $I$.
\end{definition}
\begin{prop}
  Up to an affine transformation of the independent variable, every
  SL-OPS belongs to one of the three types shown above.
\end{prop}
\begin{proof}
  Up to an affine change of variable, the second-order coefficient of
  an exceptional operator takes one of the following forms:
  \[ 1, z, z^2, 1+z^2, 1-z^2.\] Applying \eqref{eq:Wdef} and
  \eqref{eq:natqr}, we see that cases 1,2, and 5 correspond to weights
  of Hermite, Laguerre, and Jacobi type, respectively.  It therefore
  suffices to rule out the remaining possibilities.  These correspond
  to, respectively, weights of the following form:
  \begin{align*}
    W(z) &= \frac{z^a e^{\frac{b}{z}} }{\eta(z)^2},\\
    W(z) &= \frac{e^{a \arctan(z)} (1+z^2)^b}{\eta(z)^2},
  \end{align*}
  where $a,b\in \Rset$ are real constants.  By inspection, there does
  not exist a choice of constants or an interval $I\subset\Rset$ such
  that of these forms can  satisfy requirement (ii) in the definition
  of a SL-OPS.
\end{proof}

The analysis of the regularity of the exceptional weight amounts to
studying the range of parameters and the combination of Darboux
transformations such that $\eta(z)$ has no zeros on $I$, and such that
the classical portion of the weight is integrable on $I$. For the case
of exceptional Hermite polynomials, this was done in
\cite{Gomez-Ullate2014,Duran2014}, for exceptional Laguerre
polynomials in \cite{Duran2014a,Duran2015b}, and for exceptional
Jacobi polynomials in \cite{Duran2015Jacobi}.

Applying \eqref{eq:bilinear} with $p(z) = 1, z, 1-z^2$, respectively,
we arrive at the following bilinear relations for the exceptional
polynomials associated to the above 3 classes of SL-OPS:
\begin{align}
  \label{eq:xhermite}
  &(\eta \hat H_k'' - 2\eta' \hat H_k' + \eta'' \hat H_k) -2z (
  \eta \hat H_k' - \eta' \hat H_k) + 2(k-m)\, \eta \hat H_k = 0\\
  \label{eq:xlaguerre}
  & z( \eta \hat L_k'' - 2 \eta' \hat L_k' + \eta'' \hat L_k) +(1+\alpha- z)\eta
 \hat L_k'+ (z-\alpha) \eta' \hat L_k + (k-m)\, \eta \hat L_k=0,\\
  \label{eq:xjacobi}
  & (1-z^2)( \eta \hat P_k'' - 2 \eta' \hat P_k' + \eta'' \hat P_k)
  +(-(2+\alpha+\beta)z+\beta-\alpha)\eta \hat P_k'+\\ \nonumber
  &\qquad + ((\alpha+\beta)z-\beta+\alpha) \eta' \hat P_k +
  (k-m)(\alpha+\beta+1+k-m)\, \eta \hat P_k=0,
\end{align}
Here, $\hat H_k(z), \hat L_k(z),\hat P_k(z)$ denote, respectively, exceptional
Hermite, Laguerre, and Jacobi polynomials of degree $k$ corresponding
to a particular choice of $\eta(z)\in \cP_m^*$, and valid for all
$k\notin \{ k_1,\ldots, k_m\}$.  Setting $m=0$ in the above equations
recovers the usual Hermite, Laguerre, and Jacobi differential
equations.  It therefore makes sense to regard \eqref{eq:xhermite}
\eqref{eq:xlaguerre} and \eqref{eq:xjacobi} as the exceptional
generalizations of these 3 classical equations.

Theorem \ref{thm:inter} states that every exceptional operator is Darboux connected to a Bochner operator, and holds for a general class of operators defined at a purely formal level. However, Theorem \ref{thm:XOPS} is a statement about orthogonal polynomial systems, so it remains to show that the Darboux connection is guaranteed to be maintained between the more restricted class of essentially self adjoint exceptional operators that define an SL-OPS.

\begin{proof}[Proof of Theorem \ref{thm:XOPS}]
  Let $T\in \Rset\Diff(\cQ)$ be the exceptional operator associated
  with a SL-OPS.  By Theorem \ref{thm:inter}, $T$ is Darboux connected
  to a Bochner operator $\TB$ with the corresponding weights related
  by \eqref{eq:WW0}.  Since $p\in \Rset\cP_2$ is the same for both
  operators,  the $W$ and $\WB$ belong to the same class of
  weights.  In Proposition \ref{prop:Xweight}, we established that the
  polynomial $\eta(z)$ is real-valued. Therefore, the rational factor
  $\chi(z)$ in \eqref{eq:WW0} must also be real-valued, by
  \eqref{eq:pchi'}, and  $\TB$ has real coefficients.

  It remains to show that the weight parameters in $\WB$ satisfy the
  conditions in \eqref{eq:WHLJ}, so that the resulting measure has
  finite moments.  We do not claim that $\TB$ is necessarily regular, but we show next that $\TB$ is always Darboux connected to a regular Bochner operator.

  For the Hermite class, there is nothing to prove, because $p(z)=1$,
  and hence $\chi(z)$ in \eqref{eq:WW0} must be a constant.

  Let us consider the Laguerre class next.  Write
  \[ T_\alpha= z D_{zz}+ (1+\alpha-z) D_z =  (z D_z +
  1+\alpha-z)\circ D_z.\] The corresponding weight is $z^\alpha e^{-z}$.
  Performing a Darboux transformation gives
  \[ T_\alpha \mapsto D_z\circ (z D_z + 1+\alpha-z) = T_{\alpha+1} -
  1.\]
  Therefore, $T_{\alpha}$ is Darboux connected to $T_{\alpha+1}$, and
  more generally to $T_{\alpha+n}$, where $n$ is an arbitrary integer.
  Hence, even though the $\TB$ produced by Theorem \ref{thm:inter} may
  not be regular, it is Darboux connected to a regular Bochner
  operator, and hence so is $T$.

  Finally, let us consider the Jacobi class.  Write
  \[ T_{\alpha,\beta} = (1-z^2) D_{zz} + (-(2+\alpha+\beta)z+
  \beta-\alpha) D_z = \big((1-z^2) D_z  -(2+\alpha+\beta)z+
  \beta-\alpha\big) \circ D_z.\]
  Performing a Darboux transformation gives
  \[ T_{\alpha,\beta} \mapsto D_z \circ \big((1-z^2) D_z  -(2+\alpha+\beta)z+
  \beta-\alpha\big) = T_{\alpha+1,\beta+1}-2-\alpha-\beta.\]
  Therefore, $T_{\alpha,\beta}$ is Darboux connected to
  $T_{\alpha+n,\beta+n} - (2+\alpha+\beta)n$ for every integer $n$.
  By taking $n$ sufficiently large, we can ensure that
  $T_{\alpha+n,\beta+n}$ is regular.
\end{proof}

\section{Acknowledgements}

M.A.G.F. acknowledges the financial support of the Spanish MINECO
through a Severo Ochoa FPI scholarship. The work of M.A.G.F. is supported
in part by the ERC Starting Grant 633152 and the ICMAT-Severo Ochoa
project SEV-2015-0554.  The research of D.G.U. has been supported in
part by Spanish MINECO-FEDER Grants MTM2012-31714 and
MTM2015-65888-C4-3 and by the ICMAT-Severo Ochoa project
SEV-2015-0554. The research of the third author (RM) was supported in
part by NSERC grant RGPIN-228057-2009. D.G.U. would like to thank
Dalhousie University for their hospitality during his visit in the
Spring semester of 2014 where many of the results in this paper where
obtained.

\bibliographystyle{amsplainsorted}
\providecommand{\bysame}{\leavevmode\hbox to3em{\hrulefill}\thinspace}
\providecommand{\MR}{\relax\ifhmode\unskip\space\fi MR }
\providecommand{\MRhref}[2]{%
  \href{http://www.ams.org/mathscinet-getitem?mr=#1}{#2}
}
\providecommand{\href}[2]{#2}

\end{document}